\documentclass[reqno]{amsart}
\usepackage{amsfonts, amsmath, amsthm, amssymb, latexsym, xfrac, mathrsfs, mathtools, float, xcolor}

\usepackage{tikz}
\usetikzlibrary{math}

\usepackage[
  headheight=14pt,
  top=1.45in,bottom=1.35in,left=1.14in,right=1.58in,
  twoside,
  asymmetric
]{geometry}
\usepackage{hyperref}
\hypersetup{
    colorlinks = false,
    linkbordercolor = {white},
    pdfauthor=author
}

\usepackage{parskip}

\newcommand\N{\ensuremath{\mathbb{N}}}
\newcommand\R{\ensuremath{\mathbb{R}}}
\newcommand\Z{\ensuremath{\mathbb{Z}}}

\newcommand\C{\ensuremath{\mathbb{C}}}
\renewcommand\a{\mathfrak{a}}

\newcommand{\hkl}[1]{\ensuremath{\mathcal{H} #1}}
\newcommand{\sign}{\ensuremath{\text{sgn}}}

\makeatother

\usepackage[nameinlink,capitalize]{cleveref}

\makeatletter
\@namedef{subjclassname@2020}{\textup{2020} Mathematics Subject Classification}
\makeatother

\theoremstyle{plain}
\newtheorem{theorem}{Theorem}[section]
\newtheorem{corollary}[theorem]{Corollary}
\newtheorem{lemma}[theorem]{Lemma}
\newtheorem{proposition}[theorem]{Proposition}
\theoremstyle{definition}
\newtheorem{definition}[theorem]{Definition}
\theoremstyle{remark}
\newtheorem{remark}[theorem]{Remark}

\newtheorem{example}[theorem]{Example}
\numberwithin{equation}{section}

\dedicatory{Dedicated to Tom Koornwinder on the occasion of his 80th birthday}

\title{Multiresolution Analysis on Spectra of Hermitian Matrices}
\author{Lukas Langen and Margit R\"osler}
\address{Institut f\"ur Mathematik, Universit\"at Paderborn, Warburger Str. 100, D-33098 Paderborn, Germany}
\email{llangen@math.upb.de,  roesler@math.upb.de}


\thanks{The authors were supported by the  German Research Foundation (DFG), via the grants RO 1264/4-1 and SFB-TRR  358/1 2023-491392403.}

\begin{document}

\date{\today}

\begin{abstract}
    We establish a multiresolution analysis on the space $\text{Herm}(n)$ of $n\times n$ complex Hermitian matrices which is adapted to invariance  under conjugation by the unitary group $U(n).$ The orbits under this action are parametrized by the possible ordered spectra of Hermitian matrices, which constitute a closed Weyl chamber of type $A_{n-1}$ in $\mathbb R^n.$ The space $L^2(\text{Herm}(n))^{U(n)}$ of radial, i.e. $U(n)$-invariant $L^2$-functions on $\text{Herm}(n)$ is naturally identified with a certain weighted $L^2$-space on this chamber. 
    
    The scale spaces of our multiresolution analysis are obtained by usual dyadic dilations as well as generalized translations of a scaling function, where the generalized translation is a hypergroup translation which respects the radial geometry. We provide a concise criterion to characterize orthonormal wavelet bases and show that such bases always exist. They provide natural orthonormal bases of the space $L^2(\text{Herm}(n))^{U(n)}.$ 
    Furthermore, we show how to obtain  radial scaling functions from classical scaling functions on $\R^{n}$. Finally, generalizations related to the Cartan decompositions for general compact Lie groups are indicated.
\end{abstract}

\maketitle

\section{Introduction}

Suppose we are given a discrete subgroup $\Gamma\subseteq\text{GL}(V)\ltimes V$ of the affine group of a Euclidean vector space $V$ generated by translations coming from a lattice subgroup of $V$ and dilations that arise as integer powers of an expansive automorphism which leaves the lattice invariant. 

It is a classical problem in wavelet analysis  to determine whether there are functions $\psi^{1},\ldots, \psi^{r}\in L^{2}(V), $ often called a wavelet set, such that 
\begin{align}\label{wav_basis}
  \{\gamma.\psi^{i} : \gamma\in\Gamma,\ 1\leq i\leq r\} 
\end{align}
constitutes an orthonormal Hilbert basis of $L^{2}(V)$. Here $L^2(V)$ is with respect to the Lebesgue measure on $V$ and $\Gamma$ acts on $L^{2}(V)$ via $\gamma.\psi(x) = \psi(\gamma^{-1}x).$ 
The standard approach to this problem is to obtain the wavelet basis \eqref{wav_basis} from a multiresolution analysis (see e.g. \cite{baggett1999generalized, madych1993some, wojtaszczyk}), and in general the size $r$ of the wavelet set depends on the determinant of the expansive automorphism. In non-Euclidean settings, such as on manifolds, 
concepts of multiresolution analysis are often less natural.  There is a broad literature on wavelet analysis and multiresolution on spheres, see e.g. \cite{FNS18} and \cite{FFP16} for a more general background about wavelet methods on manifolds.  Let us also  mention \cite{Pap11} among further  concepts of non-Euclidean multiresolution analyses and \cite{OOR06} for a non-Euclidean  (continuous) wavelet transform on rectangular matrix spaces.  In \cite{roslerRauhut}  a radial multiresolution for $SO(3)$-invariant functions on $\R^{3}$ was introduced. It was  based on the natural hypergroup translation  on the orbit space $(\mathbb R^3)^{SO(3)}\cong [0,\infty[,$ and the intimate connection of the characters of this hypergroup (certain Bessel functions) to the Tchebychef polynomials of second kind played a crucial role.
For structural reasons, this concept cannot be generalized to $SO(n)$-invariant functions on $\R^{n}$
for $n >3.$ 

In the present paper,  we consider the space $V=\text{Herm}(n)$ of Hermitian $n\times n$-matrices as a Euclidean space, with the trace form $\langle X,Y\rangle = \text{tr}(XY)$ as scalar product. It is naturally acted upon by the unitary group $U(n)$ via conjugation, i.e. $(u, X)\mapsto u.X=uXu^{-1}$. Note that these transformations are orthogonal with respect to the trace form. 
In various analytic contexts, such as random matrix theory, one is interested in the space $L^{2}(\text{Herm}(n))^{U(n)}$ consisting of  functions $f\in L^2(\text{Herm}(n))$ which are invariant under this action, i.e. depend only on the eigenvalues of their argument. It 
seems natural to exploit the additional geometric invariance also in wavelet analysis, in order to obtain a discrete wavelet decomposition 
of $L^{2}(\text{Herm}(n))^{U(n)}.$ 
Identifying $U(n)$-orbits in $\text{Herm}(n)$ with their ordered spectrum via the spectral theorem will also reduce the dimension of the underlying space from $n^{2}$ to $n$. In a closely related  way, one could consider functions on the space $\text{SHerm}(n)= \{X \in \text{Herm}(n): \text{tr}\, X =0\},$ which are radial in the sense of conjugation invariance under $SU(n).$ As we  shall describe in the appendix (Section \ref{appendix}), the setting of \cite{roslerRauhut}  just corresponds to  $SU(2)$-invariant analysis on $\text{SHerm}(2)$. 

The goal of the present paper is to generalize the concepts of \cite{roslerRauhut} to higher rank, namely to  analysis on $\text{Herm}(n)$ which is radial in the sense of $U(n)$-invariance. 
We shall introduce a radial multiresolution on $\text{Herm}(n)$ and characterize radial orthonormal wavelet bases. A key ingredient will be the definition of a generalized translation operator on $L^{2}(\text{Herm}(n))^{U(n)}$, as classical translations of $U(n)$-invariant function need not to be $U(n)$-invariant again. We identify the orbit space of the action of $U(n)$ on $\text{Herm}(n)$ with the closed cone 
$$ \overline{\a_{+}}=\{x\in\R^n : x_{1}\geq \ldots\geq x_{n}\}$$
of ordered spectra of Hermitian matrices via $\{u.X: u \in U(n)\} \mapsto \sigma(X) \in \overline{\a_{+}}.$ Here $\sigma(X)$ denotes the set of eigenvalues of $X,$ ordered by size. The cone $\overline{\a_{+}}$ is a closed  Weyl chamber of type $A_{n-1}.$ It carries a natural hypergroup structure, where the convolution of point measures $\epsilon_x\ast \epsilon_y$ is a compactly supported probability measure on $\overline{\a_{+}}$ which describes the possible spectra of sums of Hermitian matrices $X+Y$ with given spectra $\sigma(X)=x, \sigma(Y)=y.$ For generic $x$ and $y$ the measure $\epsilon_x\ast \epsilon_y$ is absolutely continuous with respect to the Lebesgue measure in a certain affine plane in $\R^{n}$ with an explicit formula for the density. This is a  consequence of  results in \cite{graczykSawyer_complex},    see also the survey \cite{graczykSawyer_survey}. The generalized translation of suitable functions on $\overline{\a_{+}}$ is then given by 
$T_xf(y) \coloneqq \epsilon_x\ast \epsilon_y(f),$ and harmonic analysis of $L^2(\text{Herm}(n))^{U(n)}$ will play out as harmonic analysis on the $L^{2}$-space $L^{2}(\overline{\a_{+}}, \omega)$ of this hypergroup, where
$ \omega(x) = \prod_{i<j} |x_j-x_j|^2.$  
This is due to the fact that by the Weyl integration formula, the classical Fourier transform of $U(n)$-invariant functions on $\text{Herm}(n)$ coincides with a  Hankel transform with respect to the spherical functions of the Cartan motion group $U(n)\ltimes\text{Herm}(n),$ which are multivariate Bessel functions.  
The concept of a radial multiresolution in $\text{Herm}(n)$ will thus be that of a multiresolution in $L^2(\overline{\a_{+}}, \omega)$. The scale spaces $(V_{j})_{j\in\Z}$ are obtained by (classical) dyadic dilations from $V_{0}$, which is in turn spanned by generalized translations along lattice points of a so-called radial scaling function $\phi\in L^2(\overline{\a_{+}}, \omega)$. 
Similarly to \cite{roslerRauhut}, and in contrast to classical notions of multiresolution analysis, while still being characterized by a two-scale relation, the scaling function $\phi$ itself is not contained in $V_{0}$ and the scale spaces are not translation invariant with respect to the generalized translation.
While the authors in \cite{roslerRauhut} explicitly construct a set of orthonormal wavelets from a given multiresolution consisting (due to dimensionality reasons) of only one wavelet, this approach is not feasible in our higher-rank situation as the number of wavelets needed will grow with $n$.
Instead, we will give a concise criterion characterizing orthonormal wavelet bases of $L^2(\overline{\a_{+}},\omega)$ by checking whether a certain matrix-valued function is almost everywhere unitary. This result is analogous to a well-known criterion in classical wavelet theory (\cite{wojtaszczyk}) and will allow us to show that given a multiresolution analysis, associated orthonormal wavelet bases always exist and that they in fact consist of $2^{n}-1$ wavelets. 
 As functions on $\text{Herm}(n)$, these wavelets reflect the underlying radial symmetry and require reduced computational effort as classical multiresolution on the vector space $\text{Herm}(n)$ would need $2^{n^{2}}-1$ wavelets.
We are furthermore able to relate radial scaling functions to classical permutation-invariant scaling functions on $\R^{n}$, yielding a simple example of a radial wavelet basis analogous to the classical Shannon wavelets.

The paper is organized as follows: In Section \ref{radial} we recall facts about radial analysis on $\text{Herm}(n)$. We will introduce the generalized translation operators on $L^2(\overline{\a_{+}}, \omega)$ and study their properties. This will be crucial for Section \ref{mra}, where we introduce a radial multiresolution analysis on $\text{Herm}(n)$. In Section \ref{orthoWave}, orthonormal wavelet bases are discussed and Section \ref{scaling} then describes how to obtain radial scaling functions from classical scaling functions in $\mathbb R^{n}$. Finally, Section \ref{appendix} is dedicated to the discussion of generalizations of the previous results. In fact,
many arguments remain valid upon replacing $U(n)$ with an arbitrary connected compact Lie group, the space $V=\text{Herm}(n)$ by $\mathfrak{p}=i\text{Lie}(K)\subseteq\C\otimes_{\R}\text{Lie}(K),$ and by considering the adjoint action of $K$ on $\mathfrak{p}$. In this setting one can see the results of \cite{roslerRauhut} as a special case of rank $1$, as the situation corresponds to that of $SU(2)$ acting on $\text{SHerm}(2)$. However, in the general higher-rank case, it would  not be clear how to obtain radial scaling functions from classical scaling functions.


\section{Radial Analysis and Generalized Translation}\label{radial}

The unitary group $U(n),\ n\geq 2$ acts naturally on Hermitian matrices $\text{Herm}(n)$ by conjugation. 
By the spectral theorem, the orbit space $\text{Herm}(n)^{U(n)}$ of this action can be  (actually topologically) identified with the closed cone
$$ \overline{\a_{+}}:= \{x\in \mathbb R^n: x_1 \geq \ldots \geq x_n\}$$
via $U(n).X \mapsto \sigma(X),$ where $\sigma(X)\in \mathbb R^n$ denotes the ordered spectrum of $X$.
We note that $U(n)\ltimes \text{Herm}(n)$ is the Cartan motion group 
of the complex Lie group $G=GL_n(\mathbb C)$ with maximal compact subgroup $K=U(n),$ coming from the Cartan decomposition $\mathfrak{g}= \mathfrak{k}\oplus \mathfrak{p}$ with $\mathfrak{p} = \text{Herm}(n).$ 
$GL_n(\mathbb C)$   belongs to the so-called Harish-Chandra class (cf. e.g. \cite{gangolliVaradarajan}), but in contrast to $SL_n(\mathbb C)$ it is not semisimple. The space $\a = \{\text{diag}(x_1, \ldots, x_n): x_i \in \mathbb R\},$ which is a maximal abelian subspace of $\mathfrak p$, naturally identifies with $\mathbb R^n$ via $\text{diag}(x_1, \ldots, x_n)\mapsto (x_1, \ldots, x_n).$ Upon this identification, the set 
 $\,\a_{+} = \{x\in \mathbb R^n: x_1 > \ldots > x_n\}\,$ is the positive  Weyl chamber corresponding to the positive subsystem $$\Sigma_+= 
\{e_{i}-e_{j} : 1\leq i < j \leq n\}$$ of the root system $A_{n-1}$ in $\mathbb R^n.$ This explains the above notion for the orbit space $\text{Herm}(n)^{U(n)},$ and we shall often identify elements from $\mathbb R^n$ with diagonal matrices in the above  way.

Recall the Weyl integration formula (\cite{faraut_2008}, Thm. 10.1.4; more generally \cite{gangolliVaradarajan}, eqn.(2.4.22)) which states that
\begin{align*}
    \int_{\text{Herm}(n)} F(X) \,dX \,= \,c\int_{\overline{\a_{+}}}\int_{U(n)}F(uxu^{-1}) du\ \omega(x) dx
\end{align*}
for $F\in C_c(\text{Herm}(n))$. Here, $c>0$ is a constant independent of $F$ and
\begin{align*}
    \omega(x) = \prod_{i<j} (x_{i}-x_{j})^{2}. 
\end{align*}
We thus obtain an isometric isomorphism 
\begin{align*}
    \Phi\colon L^{2}(\text{Herm}(n))^{U(n)}\to L^{2}(\overline{\a_{+}}, c\omega),\quad F\mapsto f \text{ with  } f(\sigma(X)) := F(X).
\end{align*}
 As the action of $U(n)$ on $\text{Herm}(n)$ is via orthogonal transformations, the Euclidean Fourier transform of a radial function $F\in C_c(\text{Herm}(n))^{U(n)}$  is again radial and becomes a Hankel transform of  $f$: for $Y \in \text{Herm}(n)$ with $y = \sigma(Y),$
\begin{align*}
    \widehat F(Y) \coloneqq &\frac{1}{(2\pi)^{n^2/2}}\int_{\text{Herm}(n)} F(X) e^{-i\langle X,Y\rangle} dX \\ 
   = &\frac{c}{(2\pi)^{n^{2}/2}}\int_{\overline{\a_{+}}} f(x)J(x,-iy)\, \omega(x) dx \,\eqqcolon \mathcal Hf(y)
\end{align*}
with the Bessel functions
\begin{equation}\label{HIZ}
  J(y,z) \coloneqq \int_{U(n)}e^{\text{tr}(uyu^{-1}z)} du \quad (y, z\in \mathbb C^n).
  \end{equation}
  This is the Harish-Chandra integral for the spherical functions of 
  the Euclidean-type symmetric space $U(n)\ltimes \text{Herm}(n)/U(n)$, c.f. \cite{harish1957differential}. By the Harish-Chandra-Itzykson-Zuber formula (see \cite{itzyksonZuber} as well as \cite{mcswiggen} for a recent overview), it can also be written as  
 $$ J(y,z)\,= \,\frac{\prod_{k=1}^{n-1}k!}{\pi(y)\pi(z)}\sum_{w\in\mathcal S_{n}} \sign(w) e^{\langle wy,z\rangle}.$$

 Here $\sign(w)$ denotes the sign of $w,$ the scalar product $\langle\,.\,,\,.\,\rangle$ is extended in a bilinear way to $\mathbb C^n$, and
\begin{align*}
    \pi(z) := \prod_{i< j}(z_{i}-z_{j}) 
\end{align*}
denotes the fundamental alternating polynomial, also known as the Vandermonde determinant. Thus we get a Plancherel theorem extending 
the Hankel transform $\mathcal H$ to a unitary operator on $L^2(\overline{\a_{+}}, \omega).$ 
  Note that the Bessel function $J$ satisfies $J(\lambda x,z) = J(x,\lambda z)$ for $\lambda \in \mathbb C$ and  that it is $\mathcal S_{n}$-invariant in both arguments. We may therefore extend $\hkl f$ to an $\mathcal S_{n}$-invariant (i.e. symmetric) function on $\R^{n}$ whenever  convenient. 

Usual translations $F(\cdot - y)$ of radial functions on $\text{Herm}(n)$ need not be radial again. We therefore define generalized translations for $f\in C_c(\overline{\a_{+}})$ by averaging with respect to the $U(n)$-action:
\begin{align*}
     T_{x}f(y) \coloneqq \int_{U(n)} f(\sigma(x+uyu^{-1}))\, du \, \eqqcolon (\epsilon_x \ast \epsilon_y)(f)
      \quad (x,y \in \overline{\a_{+}}).
     \end{align*}
Here $\epsilon_x$ denotes the point measure in $x\in \mathbb R^n, $ which is identified with the corresponding diagonal matrix. Note  that $\epsilon_x \ast \epsilon_y$ defines a compactly supported 
probability measure on $\overline{\a_{+}}$ and that $\epsilon_y \ast \epsilon_x = \epsilon_x \ast \epsilon_y.$ As an immediate consequence of Fubini's theorem and the left-invariance of the Haar measure we obtain the product formula
$$(\epsilon_{x}*\epsilon_{y})J(\cdot, z) = J(x, z)J(y, z). 
$$
This just says that when considered as functions on $\mathbb R^n$, the functions $J(\,.\,,z), \, z \in \mathbb C^n$ are spherical functions of the Gelfand pair $(U(n) \ltimes \text{Herm}(n),U(n)).$ Indeed, all spherical functions are of this form, which follows e.g. from \cite{wolf2006spherical}, Theorem 4.4, and the bounded spherical functions are 
 those with $z \in  i\mathbb R^n.$
Moreover, for radial $F, G$ on $\text{Herm}(n),$ 
\begin{align}\label{adjungiertenrel} \int_{\overline{\a_{+}}} T_yf(x)g(x) \omega(x)dx = &\int_{\text{Herm}(n)} F(X+y)G(X)dX\,  \notag \\
  =\int_{\text{Herm}(n)} &F(X)G(X-y)dX\, = \, \int_{\overline{\a_{+}}} f(x)T_{\overline y}g(x) \omega(x)dx
 \end{align}
with $\overline y = -(y_n, \ldots, y_1),$ provided the integrals exist. 
In a similar way, it is checked that $ \|T_y f\|_{2, \omega} \leq \|f\|_{2, \omega}\,$. Hence the translation operators $T_y$ extend to norm-decreasing linear operators on $L^2(\overline{\a_{+}}, \omega).$ 

Together with the above product formula, relation \eqref{adjungiertenrel}  implies that for $f \in L^2(\overline{\a_{+}}, \omega), $
\begin{equation}\label{Hankelid} \mathcal H(T_yf)(x) = J(x,iy) \mathcal Hf(x).\end{equation}

Indeed, the above convolution $\epsilon_{x}*\epsilon_{y}$ of point measures $\epsilon_{x}$ and $\epsilon_y$ defines a commutative orbit hypergroup structure on $\overline{\a_{+}}$ in the sense of \cite{jewett1975spaces}, Sec. 8 (where hypergroups are called convos). The neutral element is $0\in \overline{\a_{+}}$, the involution is given by $x \mapsto \overline x$ as  defined above and $ \omega(x)dx$ is a Haar measure. For details on hypergroups related to motion groups as in the present setting, see also \cite{roslerVoitAngers}. 

In Section \ref{mra} 
we shall need some more refined information on the measures $\epsilon_{x}*\epsilon_{y}$. Indeed, a combination and adaption of results from Helgason \cite{helgasonGroups}, Graczyk, Sawyer \cite{graczykSawyer_complex}
and Graczyk, Loeb \cite{graczykLoeb} for the semisimple case shows that  for  $x, y\in \a_{+}$, the measure $\epsilon_{x}*\epsilon_{y}$ is  absolutely continuous with respect to the Lebesgue measure on a certain affine plane in $\R^{n},$ and there is an explicit  formula for the density. To make this precise, we orthogonally decompose $\R^{n}=\R^{n}_{0}\oplus \R\,\underline 1$ with
\begin{align*}
    \R^{n}_{0} = \{x\in\R^{n} : x_{1}+\ldots+x_{n}=0\}, \quad \underline 1 = (1, \ldots, 1)
\end{align*}
and denote by $x^{0}$ and $x^{1}$ the orthogonal projections of $x\in\R^{n}$ onto $\R^{n}_{0}$ and $\R\underline 1$, respectively. We further put 
$\mathbb C_0^n \coloneqq\{z\in \mathbb C^n: z_1 + \ldots +z_n=0\}$ and
$$\a_{+0}\coloneqq\R^{n}_{0}\cap \a_{+}.$$

A canonical basis of $\R^{n}_{0}$ is given by the vectors
\begin{align*}
    \alpha_{i}\coloneqq e_{i}-e_{i+1},\qquad 1\leq i\leq n-1,
\end{align*}
which form the (unique) simple system in $\Sigma_+$. Denote  the remaining elements in 
$
\Sigma_+ \setminus \{\alpha_{1},\ldots,\alpha_{n-1}\}$ by $\alpha_{n},\ldots,\alpha_{q}$. Let $q\coloneqq|\Sigma_+|=\frac 1 2 n(n-1)$. 
Following \cite{graczykLoeb, graczykSawyer_complex}, we express $\alpha_{k}=\sum_{j=1}^{n-1}a_{kj}\alpha_{j}$ for $k=n,\ldots, q$ in the basis above and define the subset $\Delta(y_{1},\ldots, y_{n-1})\subseteq\R^{q-n+1}$ via
\begin{align*}
    (y_{n},\ldots, y_{q})\in \Delta(y_{1},\ldots, y_{n-1})\quad\Longleftrightarrow\quad& y_{n},\ldots, y_{q}\geq 0\quad\text{and}\\ 
    &\ \sum_{k=n}^{q} y_{k} a_{kj}\leq y_{j},\quad j=1,\ldots, n-1.
 \end{align*}   
 Note that $a_{kj}\geq 0$ for all $k,j$, since positive roots are non-negative linear combinations of roots from the associated simple subsystem (\cite{Hum90}). Therefore 
 $\Delta(y_{1},\ldots, y_{n-1})\subseteq\R^{q-n+1}$ is compact. 

For $n \geq 3,$ we further define a function $T\colon\R^{n}_{0}\to\mathbb R$ by
\begin{align*}	
    T(y_{1}\alpha_{1}+\ldots+y_{n-1}\alpha_{n-1}) := \int_{\Delta(y_{1},\ldots, y_{n-1})}\ dy_{n}\ldots dy_{q}.
\end{align*}
If $n=2$, then we define $T$ to jump from $1$ inside $\overline{\a_{+0}}$ to $0$ on $\R^{2}_{0}\setminus\overline{\a_{+0}}$. 

In order to describe some further properties of $T$, we introduce in $\mathbb R_0^n$ the dual cone of $ \overline{\a_{+0}}$, which is given by
\begin{align*}
    {}^{+}\a = \{x\in\R^{n}_{0} : x=\sum_{j=1}^{n-1}c_{j}\alpha_{j},\,\, c_{j}> 0\} \subseteq\R^{n}_{0}.
\end{align*}

We first state some general facts which will be useful later on. For $x\in \mathbb R^n$, we denote by
\begin{equation*} C(x):=\operatorname{conv}(\mathcal{S}_{n}.x)\end{equation*}
 the convex hull of the $\mathcal S_n$-orbit of $x$.
\begin{lemma}\label{lem:eigenschaftenOrdnung}
    \begin{enumerate}
        \item If $x\in\a_{+},$ then $\,x-wx\in\overline{{}^{+}\a}$ for all $w\in\mathcal S_{n}$.
        \item  If $x\in\overline{\a_{+}}$ and $h\in C(x)$, then $x-h\in\overline{{}^{+}\a}$.
        \item  $\rho\in\a_{+0}$.
    \end{enumerate}
\end{lemma}

\begin{proof} See \cite{helgasonGroups}, Ch. IV, Lemma 8.3. and observe that $x-wx \in \mathbb R_0^n$ for $x\in \mathbb R^n.$
\end{proof}	

We shall need the following facts about $T,$ established by Graczyk and Loeb:

\begin{lemma}[\cite{graczykLoeb}, Prop. 2]\label{lem:eigenschaftenT}
    $T$ is supported in $\overline{{}^{+}\a}$ and continuous and nonnegative on $\overline{{}^{+}\a}$. If $x, y \in\overline{{}^{+}\a}$ with $y-x\in\overline{{}^{+}\a}$, then $T(x)\leq T(y)$.
\end{lemma}

For regular arguments $x,y\in \a_{+}$, the convolution product $\epsilon_x\ast \epsilon_y$ is now described as follows in terms of $T$: 

\begin{proposition}\label{prop:dichteTranslation}
For $x, y\in \a_{+}$, the measure $\epsilon_{x}*\epsilon_{y}$ is 
       absolutely continuous with respect to the Lebesgue measure on $x^{1}+y^{1}+\overline{\a_{+0}} \subseteq x^{1}+y^{1}+\R^{n}_{0}$. The density is given by
    \begin{align*}
      k(x, y, h)=\frac{\pi(\rho)\pi(h)}{\pi(x)\pi(y)} \sum_{v, w\in\mathcal S_{n}} \operatorname{sgn}(v)\operatorname{sgn}(w)\,T(vy+wx-h).
    \end{align*}
    Here $\rho \coloneqq\frac 1 2 \sum_{i < j}(e_{i}-e_{j}) = \frac{1}{2}(n-1, n-3, \ldots, -n+1)\,\,$ denotes the Weyl vector.
\end{proposition}

\begin{proof}

  Throughout the paper, we denote the pushforward of a measure $\mu$ under a measurable mapping $\varphi$ by $\varphi_\ast(\mu).$ As $U(n) = H \ltimes SU(n)$ with $ H= \{a_z = \text{ diag}(z,1, \ldots,1), \,z\in  U(1)\},$ Weyl's integration formula gives
  $$(\epsilon_x \ast \epsilon_y)(f) = \int_{U(1)}\int_{SU(n)} f(\sigma(x+a_z kyk^{-1} a_z^{-1}))dk\, dz \,= \int_{SU(n)} f(x^1+y^1 + \sigma(x^0+ ky^0k^{-1}))dk$$
  for $x,y \in \overline{\a_{+}}$. This shows that 
     \begin{align*}
         \epsilon_{x}*\epsilon_{y}=[\xi\mapsto \xi + x^{1}+y^{1}]_{*}(\epsilon_{x^{0}}*_0\epsilon_{y^0}),
    \end{align*}
    where
    \begin{align*}
        (\epsilon_{x^{0}}*_0\epsilon_{y^{0}})(g) \coloneqq \int_{SU(n)} g(\sigma(x^{0}+ky^{0}k^{-1})) dk
    \end{align*}
    for $g\in C(\overline{\a_{+0}})$.
The measure $\epsilon_{x^{0}}*_0\epsilon_{y^{0}}$ is a compactly supported probability measure on $\overline{\a_{+0}}$ which describes the orbit hypergroup convolution on the space $\text{SHerm}(n)^{SU(n)},$ where $\text{SHerm}(n)$ are the Hermitian matrices of trace $0$, acted upon by $SU(n)$ via conjugation. Observe that for $z\in \mathbb C_0^n,$
$$ J(x^0, z) = \int_{SU(n)} e^{\text{tr}(kx^0k^{-1}z)}dk$$ 
and 
 \begin{align*}
        (\epsilon_{x^{0}}*_0\epsilon_{y^{0}})(J(\cdot, z)) = J(x^{0},z)J(y^{0},z).
    \end{align*}
       The functions  $\psi_{z}=J(\cdot, iz)|_{\R^{n}_{0}},\ z\in\C^{n}_{0}$ are just the spherical functions of  the Cartan motion group $SU(n)\ltimes \text{SHerm}(n).$ They are related to the spherical functions $\phi_{z}$ of the Riemannian symmetric space $SL(n, \C)/SU(n)$ of complex type via
    \begin{align*}
        \psi_{z}(x) = \frac{\widetilde\Delta(x)}{\pi(x)}\phi_{z}(e^{x})
    \end{align*}
    with $$\widetilde\Delta(x)=\prod_{i< j}\sinh(x_{i}-x_{j}),$$ see \cite{helgasonGroups}, Ch. IV, Prop. 4.10. and Ch. II, Thm. 3.15,  or the nice presentation in  \cite{BSO05}, Sect. 9 (where the definition of $\rho$ differs by a factor $2$ from ours).   Now the claim follows from \cite{graczykSawyer_complex}, Thm. 2.1 and Prop. 3.1, which  contain an explicit expression for the density in the product formula of the $\phi_{z}$ and thus for $\epsilon_{x^{0}}*_0\epsilon_{y^{0}}$. 
    \end{proof}
    
    \begin{remark}
   Note that in \cite{graczykSawyer_complex}, the authors  work with  $\mathcal S_{n}$-invariant measures on $\a= \R_0^{n},$
     while we work with  measures supported in $\overline{\a_{+0}}$. This results in a factor $1/|\mathcal S_n|$ in \cite{graczykSawyer_complex}, Prop. 3.1. We further note that in \cite{graczykSawyer_complex} the authors missed a factor $2^{|\Sigma_+|}= 2^{\frac 1 2 n(n-1)}$ in equation (3), which  
     leads to a missing factor $2^{-|\Sigma_+|}$ in Prop. 3.1. 
\end{remark}

\begin{remark} Note that with respect to the decomposition  $\mathbb R^n = \mathbb R_0^n \oplus \mathbb R\underline 1$ we 
can write $\epsilon_x\ast\epsilon_y$ as a product measure:
 $$ \epsilon_x\ast\epsilon_y  = (\epsilon_{x^0}\ast_0\epsilon_{y^0})\otimes \epsilon_{x^1 + y^1}\,.$$	
\end{remark}

\begin{lemma}\label{supp}
    Suppose $x, y\in\a_{+}$ such that $y+C(x)\subseteq\a_{+}$. Then 
    $$
        \operatorname{supp}\,(\epsilon_{x}*\epsilon_{y}) \subseteq \,y+C(x).         
         $$
    \end{lemma}
\begin{proof}
    This follows from \cite{graczykSawyer_complex}, Cor. 2.2. combined with the fact that $\epsilon_{x}*\epsilon_{y}=(\epsilon_{x^{0}}\ast_0 \epsilon_{y^{0}})\otimes \epsilon_{x^1+y^1}.$  Note again that the authors there work with an $\mathcal S_{n}$-invariant formulation. As we assumed $y+C(x)\subseteq\a_{+}$, the pushforward measures $w_{*}(\epsilon_{x}*\epsilon_{y})$ have disjoint supports, resulting in the expression above. 
\end{proof}

Besides the norm-decreasing translation operators $T_y$ on $L^2(\overline{\a_{+}},\omega),$ we will need  dilation operators in order to establish a radial multiresolution. These can be defined in the usual way: Note that $\omega$ is homogeneous of degree $2|\Sigma_+|$ and put 
$$ m\coloneqq \operatorname{dim}_{\mathbb R} \text{Herm}(n) + 2|\Sigma_+|\, = 2n^2 -n.$$
Then one easily calculates that for $a>0,$
\begin{align}\label{dilation}
    D_{a} f(x) \coloneqq a^{- \frac{m}{2}} f(\tfrac 1 a x)
\end{align}
defines a unitary operator $D_{a}$ on $L^2(\overline{\a_{+}}, \omega)$ which satisfies
\begin{align}\label{eqn:hklVsDilation}
    \hkl(D_{a}f)(x) = D_{1/a}(\hkl f)(x). 
\end{align}

Finally, we will need an orthonormal basis of $L^2(\overline{\a_{+}},\omega)$ which behaves nicely with respect to the Hankel transform. To this end, we consider the Schur functions
$(s_\lambda)_{\lambda\in  P_{+}} $ in $n$ variables, c.f. \cite{bump} or \cite{faraut_2008}.  These are $\mathcal S_n$-invariant trigonometric polynomials which   are indexed by the set 
$$  P_{+} \coloneqq \{\lambda \in \mathbb \Z^{n}: \lambda_1 
\geq \ldots \geq \lambda_n\}  $$ 
and given by
$$s_\lambda(x) = \frac{1}{\Delta(x)}\sum_{w\in \mathcal S_n} \sign(w)e^{i\langle  \lambda+\rho, wx\rangle} \, =\, \frac{A_{\lambda+\delta}(e^{ix})}{A_\delta(e^{ix})}, \quad x\in \mathbb R^n$$
with $$ \Delta(x) = \prod_{i<j}(e^{i(x_{i}-x_{j})/2}-e^{-i(x_{i}-x_{j})/2})  
=\sum_{w\in\mathcal S_{n}} \sign(w)e^{i\langle \rho, w x\rangle},$$ $$ A_\lambda(e^{ix}) \coloneqq \det\bigl(e^{i\lambda_jx_k}\bigr)_{1\leq j,k\leq n} \quad \text{and}\quad\delta = (n-1, n-2, \ldots ,0).$$ 
Note that $\,P_+ = \Z^{n}\cap \overline{\a_{+}} $ and that \begin{align*}
  \Delta(x) = e^{-i\frac{n-1}{2}\langle x,\underline 1\rangle} A_{\delta}(e^{ix}), 
\end{align*}
so that in particular $|\Delta|$ is $\mathcal S_{n}$-invariant and $I$-periodic for $$I=2\pi\mathbb Z^n.$$
It is well-known that the $s_\lambda\,,\,\lambda\in P_+$ are the Weyl characters of $U(n)$, considered as functions on its maximal torus $T \cong \mathbb T^n$ with $\mathbb T = \mathbb R/2\pi\mathbb Z\,.$ By the Schur orthogonality relations and the Peter-Weyl theorem, they form an orthogonal  basis of the Hilbert space 
$$ X\coloneqq L^2(\mathbb T^n, |\Delta|^2)^{\mathcal S_n} $$
 consisting of those functions $f\in L^2(\mathbb T^n, |\Delta|^2)$ which are $\mathcal S_n$-invariant. More precisely, 
the Schur functions satisfy
$$ \frac{1}{n!}\int_{\mathbb T^n} s_\lambda(x)\overline{s_\mu(x)}\,|\Delta(x)|^2 dx = \delta_{\lambda, \mu}.$$
For details, see \cite{bump}, Thm. 22.2 ff. together with Thm. 36.2, or \cite{faraut_2008}, Section 11. We consider the following renormalization of Schur functions: 
$$ S_\lambda(x)\coloneqq \frac{1}{i^{|\Sigma_+|}\sqrt{n!}}\, s_\lambda(x), \quad \lambda \in P_+\,.$$
The $(S_\lambda)_{\lambda\in P_+}$ form an orthonormal basis of $X$. 
Moreover, in view of formula \eqref{HIZ}, they are connected with the Bessel function $J$ via
\begin{align}\label{S_Bessel} S_\lambda(x) = M_\lambda\frac{\pi(x)}{\Delta(x)} J(x,i(\lambda+\rho)) \>\> \text{ with } \>\> M_\lambda = 
 \frac{\pi(\lambda+\rho)}{\sqrt{n!}\,\prod_{k=1}^{n-1}k!}\,.\end{align}

This connection and the periodicity of the $S_\lambda$ will be of crucial importance for our constructions, together with the observation that $\pi^2 = \omega$ on $\mathbb R^n$, where $\omega$ occurs as the density of the Haar measure of the hypergroup $(\overline{\a_{+}}, \ast)$ discussed before. 
Sometimes it will be convenient to work on the fundamental domain
$$ D = [0, 2\pi[^n $$
of the torus $\mathbb T^n= \mathbb R^n/I$ and  use the identification
\begin{equation}\label{identS} X \cong \{\alpha: \mathbb R^n\to \mathbb C: \alpha\vert_{D} \in L^2(D, |\Delta|^2)^{\mathcal S_n}, \alpha(x+q) = \alpha(x) \> \text{a.e.} \> (\forall q\in I)\}.\end{equation}

Note  that the space $X$ is the analogue of the space $S$ in \cite{roslerRauhut}, where a radial multiresolution analysis in $\mathbb R^3$ was constructed.

For abbreviation we further introduce the notation
\begin{align*}
    T^{(\lambda)} \coloneqq T_{\lambda+\rho} \quad (\lambda\in P_+).
\end{align*}
As a consequence of identity \eqref{Hankelid} for the Hankel transform, we have 
\begin{align}\label{eqn:radial:analogon43}
    \hkl(M_{\lambda}T^{(\lambda)}f)(\xi) = \frac{\Delta(\xi)S_{\lambda}(\xi)}{\pi(\xi)}\hkl f(\xi)
\end{align}
for $f\in L^2(\overline{\a_{+}}, \omega)$.


\section{Radial Multiresolution Analysis in $\text{Herm}(n)$}\label{mra}

We start with the definition of a radial multiresolution analysis (MRA) on $\text{Herm}(n)$. It is a higher-rank anlogue of the concept of a radial MRA in 
$ \mathbb R^3$ introduced in \cite{roslerRauhut}. 

\begin{definition}\label{def:mra:mra}\ 
  We call a (bilateral) sequence $(V_{j})_{j\in\Z}$ of closed linear subspaces of $L^{2}(\overline{\a_{+}}, \omega)$ a (dyadic) radial MRA in $\text{Herm}(n),$ if it satisfies the following properties:
    \begin{enumerate}
        \item $V_{j}\subseteq V_{j+1}$ for all $j\in\Z$,
        \item  $\bigcap_{j=-\infty}^{\infty} V_{j} = \{0\}$,
        \item  $\bigcup_{j=-\infty}^{\infty}V_{j}$ is dense in $L^{2}(\overline{\a_{+}}, \omega)$,
        \item  $f\in V_{j}$ if and only if $f(2\,\cdot) \in V_{j+1}$,
        \item  there exists a function $\phi\in L^{2}(\overline{\a_{+}}, \omega)$ such that
            \begin{align*}
                B_{\phi} \coloneqq \left\lbrace M_{\lambda}T^{(\lambda)}\phi : \lambda\in P_{+} \right\rbrace  
            \end{align*}
            is a Riesz basis of $V_{0}$, i.e. $\text{span}\, B_{\phi}$ is dense in $V_{0}$ and there exist constants $A, B > 0$ such that
            \begin{align*}
                A\|\alpha\|_{2}^{2} \,\leq \Bigg\|\sum_{\lambda\in P_{+}}\alpha_{\lambda}M_{\lambda}T^{(\lambda)}\phi\,\Bigg\|_{2}^{2}\leq B\|\alpha\|_{2}^{2} 
            \end{align*}
            for all $\alpha = (\alpha_{\lambda})_{\lambda\in P_{+}}\in\ell^{2}( P_{+})$.
    \end{enumerate} 

    The function $\phi$ is called \emph{scaling function} for the MRA $(V_{j})_{j\in\Z}$.
\end{definition}

\begin{definition}
For $\phi \in  L^2(\overline{\a_{+}}, \omega)$ we introduce the function
\begin{align*}
    P_{\phi}(\xi) \coloneqq \frac{1}{n!}\sum_{q\in I} |\hkl{\phi}(\xi+q)|^{2}\qquad (\xi\in\R^{n}),
\end{align*}
which is $\mathcal S_{n}$-invariant and $I$-periodic. 
\end{definition}

\begin{proposition}\label{prop:mra:analogonProp42}
    Let $\phi\in L^2(\overline{\a_{+}}, \omega)$ and $A, B > 0$. Then
    \begin{align*}
        A\|\alpha\|_{2}^{2} \,\leq \Bigg\|\sum_{\lambda\in P_{+}}\alpha_{\lambda}M_{\lambda}T^{(\lambda)}\phi\,\Bigg\|_{2}^{2} \leq B\|\alpha\|_{2}^{2} \quad\text{for all}\ \alpha\in\ell^{2}( P_{+}) 
    \end{align*}
    if and only if
    \begin{align*}
        A\leq P_{\phi}(\xi) \leq B \quad \text{for almost all}\ \xi\in\overline{\a_{+}}. 
    \end{align*}
\end{proposition}
\begin{proof}
    Let $\alpha\in\ell^{2}( P_{+})$ be a finite sequence. Define
    \begin{align*}
      \widetilde{\alpha}\coloneqq \sum_{\lambda\in P_{+}}\alpha_{\lambda}S_{\lambda} \in X.
    \end{align*}
    Using the Plancherel theorem for the Hankel transform and \cref{eqn:radial:analogon43} we obtain
    \begin{align*}
        \Bigg\|\sum_{\lambda\in P_{+}} \alpha_{\lambda}M_{\lambda}T^{(\lambda)}\phi\,\Bigg\|_{2}^{2} &= \Bigg\|\sum_{\lambda\in P_{+}}\alpha_{\lambda}\frac{\Delta S_{\lambda}}{\pi}\,\hkl\phi\,\Bigg\|_{2}^{2}\\ 
        &=\int_{\overline{\a_{+}}} \Bigg\vert \sum_{\lambda\in P_{+}}\alpha_{\lambda} S_{\lambda}(\xi)\Bigg\vert^{2} |\hkl\phi(\xi)|^{2} \frac{\omega(\xi)}{|\pi(\xi)|^{2}} |\Delta(\xi)|^{2} d\xi\\
        &= \frac{1}{n!}\int_{\R^{n}} |\widetilde{\alpha}(\xi)|^{2}\, |\hkl\phi(\xi)|^{2} |\Delta(\xi)|^{2}d\xi\\
     &=\frac{1}{n!}\sum_{q\in I}\int_{\mathbb T^n} |\widetilde{\alpha}(\xi+q)|^{2} |\hkl\phi(\xi+q)|^{2} |\Delta(\xi)|^{2}d\xi\\
        &= \int_{\mathbb T^n} |\widetilde{\alpha}(\xi)|^{2} P_{\phi}(\xi)\ |\Delta(\xi)|^{2}d\xi.
    \end{align*}
    The $\{S_{\lambda} : \lambda\in P_{+}\}$ forming an orthonormal basis of $X,$ we have $$\|\alpha\|_{2}=\|\widetilde{\alpha}\|_X\,.$$ Now the assertion is immediate, since finite sequences are dense in $\ell^{2}( P_{+})$.
\end{proof}

With $A=B=1$ we obtain

\begin{corollary}\label{kor:mra:analogonKorr43}
    For $\phi\in L^2(\overline{\a_{+}}, \omega)$ the following are equivalent:
    \begin{enumerate}
    \item The set $B_{\phi}\coloneqq \{M_{\lambda}T^{(\lambda)}\phi : \lambda\in P_{+}\}$ is orthonormal in $L^2(\overline{\a_{+}}, \omega)$.
        \item  $P_{\phi} = 1$ a.e.
    \end{enumerate}
\end{corollary}

For $\phi\in L^2(\overline{\a_{+}}, \omega)$, put 
\begin{align*}
    V_{\phi} \coloneqq \overline{\text{span}\, B_{\phi}}\subseteq  L^2(\overline{\a_{+}}, \omega). 
\end{align*}
The set $B_{\phi}$ is a Riesz basis of $V_{\phi}$ if and only if the two equivalent conditions in the previous proposition are satisfied for some constants $A, B >0$. In this case we say that $\phi$ satisfies  \emph{condition (RB)}.

\begin{lemma}\label{lem:mra:analogonLemma45}
    Let $\phi\in L^2(\overline{\a_{+}}, \omega)$ satisfy (RB). Then for $f\in L^2(\overline{\a_{+}}, \omega)$ we have the equivalence
    \begin{align*}
        f\in V_{\phi}\quad \Longleftrightarrow\quad \hkl f(\xi)=\frac{\Delta(\xi)\beta(\xi)}{\pi(\xi)}\,\hkl \phi(\xi)\quad\text{with }\ \beta\in S. 
    \end{align*}
    The function $f\in V_{\phi}$ corresponding to $\beta=\sum_{\lambda\in P_{+}}\alpha_{\lambda}S_{\lambda}\in X$ with $\alpha\in\ell^{2}( P_{+})$ is given by $f=\sum_{\lambda\in P_{+}}\alpha_{\lambda}M_{\lambda}T^{(\lambda)}\phi$.
\end{lemma}

\begin{proof}
    As the $M_{\lambda}T^{(\lambda)}\phi$ form a Riesz basis for $V_{\phi}$, the functions
    \begin{align*}
        \frac{\Delta S_{\lambda}}{\pi}\ \hkl \phi = \hkl(M_{\lambda}T^{(\lambda)}\phi) 
    \end{align*}
    form a Riesz basis of $\hkl(V_{\phi})$, where we again used \cref{eqn:radial:analogon43}. This yields the assertion, since 
    for $\alpha \in \ell^2(P_+),$ 
    $$
       \mathcal H\Bigl(\sum_{\lambda\in P_{+}}\alpha_{\lambda} M_{\lambda}T^{(\lambda)}\phi\Bigr) = \frac{\Delta}{\pi}\Bigl(\sum_{\lambda\in P_{+}}\alpha_{\lambda}S_{\lambda}\Bigr)\hkl\phi.
    $$
\end{proof}

When $\phi$ is the scaling function of an MRA $\{V_j\}$, then $V_\phi= V_0$. The following corollary shows that in  contrast to the classical notion of an MRA, $V_0$ is not shift-invariant (and similarly the other scale spaces):

\begin{corollary}
    Let $(V_{j})_{j\in\Z}$ be a radial MRA. Then $f\in V_{0}$ implies $\,T^{(\lambda)} \!f\notin V_{0}$ for all $\lambda\in P_{+}$. 
\end{corollary}

\begin{proof} Recall that $\mathcal H(T^{(\lambda)} \!f) = J(\,.\,,i(\lambda+\rho)) \mathcal Hf.$ 
    But if $\beta\in X$, then $J(\cdot, i(\lambda+\rho))\,\beta\notin X$ as the functions $J(\cdot, i(\lambda+\rho))$ are not periodic. Now the previous lemma implies the assertion. 
\end{proof}

\begin{theorem}[Orthonormalization]\label{thm:mra:orthogonalization}
    Suppose that $\phi\in L^2(\overline{\a_{+}}, \omega)$ satisfies condition (RB), i.e.  there are constants $A, B>0$ such that $A\leq P_\phi(\xi) \leq B$ a.e. Define $\phi^{*}\in L^2(\overline{\a_{+}}, \omega)$ by its Hankel transform
    \begin{align}\label{renorm}
        \hkl \phi^{*}\coloneqq\frac{\hkl \phi}{\sqrt{P_{\phi}}}. 
    \end{align}
    Then $B_{\phi^{*}}=\{M_{\lambda}T^{(\lambda)}\phi^{*} : \lambda\in P_{+}\}$ forms an orthonormal basis of $V_{\phi}=V_{\phi^{*}}$.     
\end{theorem}

\begin{proof} Note first that as a consequence of condition (RB), the right side in \eqref{renorm} belongs to $L^2(\overline{\a_{+}}, \omega).$ 
    By definition $P_{\phi^{*}}=1$ a.e., so by \cref{kor:mra:analogonKorr43} it only remains to prove $V_{\phi}=V_{\phi^{*}}$. For this, it suffices to verify that $M_{\lambda}T^{(\lambda)}\phi^{*}\in V_{\phi}$ and $M_{\lambda}T^{(\lambda)}\phi\in V_{\phi^{*}}$ for all $\lambda\in P_{+}$. We employ \cref{lem:mra:analogonLemma45} and see
    \begin{align*}
        M_{\lambda}T^{(\lambda)}\phi^{*}\in V_{\phi}\quad&\Longleftrightarrow\quad \hkl(M_{\lambda}T^{(\lambda)}\phi^{*})=\frac{\Delta \beta}{\pi}\hkl\phi\quad\text{for some}\ \beta\in X.
    \end{align*}
    But
    \begin{align*}
        \hkl(M_{\lambda}T^{(\lambda)}\phi^{*})=\frac{\Delta S_{\lambda}}{\pi}\frac{\hkl\phi}{\sqrt{P_{\phi}}}. 
    \end{align*}
    Analogously we get
    \begin{align*}
        M_{\lambda}T^{(\lambda)}\phi\in V_{\phi^{*}}\quad\Longleftrightarrow\quad \hkl(M_{\lambda}T^{(\lambda)}\phi)=\frac{\Delta \widetilde\beta}{\pi}\hkl\phi^{*} = \frac{\Delta \widetilde\beta}{\pi}\frac{\hkl\phi}{\sqrt{P_{\phi}}}\quad\text{for some}\ \widetilde\beta\in X
    \end{align*}
    and
    \begin{align*}
        \hkl(M_{\lambda}T^{(\lambda)}\phi)=\frac{\Delta S_{\lambda}}{\pi}\,\hkl\phi. 
    \end{align*}
    This gives, for all $\lambda \in P_+$,  the conditions
    \begin{align*}
      \frac{S_{\lambda}}{\sqrt{P_{\phi}}}\in X\quad\text{and}\quad S_{\lambda}\sqrt{P_{\phi}}\in X.
    \end{align*}
     But these conditions are guaranteed by our assumption on $P_{\phi}$.
\end{proof}

Given a function $\phi\in L^2(\overline{\a_{+}}, \omega)$ satisfying (RB) plus some additional conditions (see \cref{prop:mra:prop48analogon} and \cref{thm:mra:schnittSkalenrme} below), we will now construct a radial MRA having $\phi$ as scaling function. Recall the unitary dilations \ref{dilation} and
define scale spaces $\{V_{j}\}_{j\in\Z}$ by
\begin{align*}
    V_{0}\coloneqq V_{\phi},\quad V_{j}\coloneqq D_{2^{-j}} V_{0}.
\end{align*}
Then property (4) of \cref{def:mra:mra} is satisfied by construction. Putting
\begin{align}\label{eqn:mra:phiJLambda}
    \phi_{j,\lambda}(\xi)\coloneqq D_{2^{-j}}(M_{\lambda}T^{(\lambda)}\phi)(\xi)=2^{\frac{jm}{2}}M_{\lambda}(T^{(\lambda)}\phi)(2^{j}\xi)\quad (j\in\Z,\,\lambda\in P_{+}), 
\end{align}
we have
\begin{align*}
    \langle \phi_{j,\lambda}, \phi_{j,\mu}\rangle = \langle \phi_{0,\lambda},\phi_{0,\mu}\rangle 
\end{align*}
since dilations are unitary. Thus $\{\phi_{j,\lambda} : \lambda\in P_{+}\}$ is a Riesz basis of $V_{j}$ with the same Riesz constants $A, B>0$ as for $\phi$. In particular,
\begin{align*}
    V_{j}=\overline{\text{span}\,\{\phi_{j,\lambda} : \lambda\in P_{+}\}}. 
\end{align*}
Moreover, if $B_{\phi}=\{\phi_{0,\lambda} : \lambda\in P_{+}\}$ is an orthonormal basis of $V_{0}$, then $\{\phi_{j,\lambda} : \lambda\in P_{+}\}$ constitutes an orthonormal basis of $V_{j}$. We shall now analyze the further required properties of \cref{def:mra:mra} in this case. We start with the condition 
that the scale spaces should be nested. 

\begin{proposition}\label{prop:mra:prop48analogon}
    For $\phi\in L^2(\overline{\a_{+}}, \omega)$ satisfying (RB) and the spaces $V_{j}$ defined above, the following statements are equivalent:
    \begin{enumerate}
        \item $V_{j}\subseteq V_{j+1}$ for all $j\in\Z$.
        \item  $V_{-1}\subseteq V_{0}$.
        \item  There exists a function $\gamma\in X$ such that $\phi$ satisfies the two-scale relation
            \begin{align*}
                \Delta(2\xi)\hkl\phi(2\xi) = \gamma(\xi)\Delta(\xi)\hkl\phi(\xi). 
            \end{align*}
    \end{enumerate}
    In this case, if we expand $\phi_{-1,0}$ in $V_0$ as $\phi_{-1,0}=\sum_{\lambda\in P_{+}}\alpha_{\lambda}\phi_{0,\lambda}$ with $\alpha=(\alpha_\lambda)_{\lambda\in P_{+}}\in\ell^{2}( P_{+})$, then $\,\gamma = c\cdot\sum_{\lambda\in P_{+}}\alpha_{\lambda}S_{\lambda}\,$ 
    with $\,c= 2^{-\frac{n^{2}}{2}}i^{|\Sigma_{+}|}\sqrt{n!}\,$. 
\end{proposition}

\begin{proof}
    The equivalence of (1) and (2) is immediate by rescaling. Suppose now $V_{-1}\subseteq V_{0}$. Then by \cref{lem:mra:analogonLemma45} shows that
    \begin{align*}
        \hkl\phi_{-1, 0} = \frac{\Delta\beta}{\pi}\hkl\phi 
    \end{align*}
    with some $\beta\in S$. On the other hand, using equations \eqref{eqn:radial:analogon43} and \eqref{eqn:hklVsDilation} we 
    calculate
    \begin{align*}
        \hkl\phi_{-1,0}(\xi) &= \hkl(D_{2}M_{0}T^{(0)}\phi)(\xi)\,=\,
          D_{1/2}\bigg(\frac{\Delta S_{0}}{\pi}\hkl\phi\bigg)(\xi)\\
   &= 2^{m/2}\, \frac{\Delta(2\xi)S_{0}(2\xi)}{\pi(2\xi)} \hkl\phi(2\xi)\,
   = \frac{2^{n^2/2}}{i^{|\Sigma_+|}\sqrt{n!}}\frac{\Delta(2\xi)}{\pi(\xi)}\hkl\phi(2\xi).
    \end{align*}
    We conclude that 
    \begin{align}\label{eqn:mra:proofStep:twoScaleRelation0}
    \Delta(2\xi) \hkl\phi(2\xi) = c\cdot\beta(\xi) \Delta(\xi)\hkl\phi(\xi)
    \end{align} 
    with $c$ as stated and obtain the desired result with $\gamma = c\beta\in X$. If $\phi_{-1, 0}=\sum_{\lambda\in P_{+}} \alpha_{\lambda}\phi_{0, \lambda}$, we employ \cref{lem:mra:analogonLemma45} and see that $\beta = c^{-1}\gamma = \sum_{\lambda\in P_{+}}\alpha_{\lambda}S_{\lambda}$.

    Conversely, assume that (3) is satisfied. 
    Performing the same calculation as above for $\phi_{-1, \lambda}$ leads to
    \begin{align*}
      \hkl\phi_{-1,\lambda}=\text{const.}\, \frac{S_{\lambda}(2\xi)}{\pi(\xi)}\Delta(2\xi)\hkl\phi(2\xi) 
    \end{align*}
    and an application of  \cref{lem:mra:analogonLemma45} shows that $\phi_{-1,\lambda}\in V_{0}$ if and only if
    \begin{align}\label{eqn:mra:proofStep:twoScaleRelationLambda}
      S_\lambda(2\xi)\Delta(2\xi)\hkl\phi(2\xi) = \gamma_{\lambda}(\xi)\Delta(\xi)\hkl\phi(\xi) 
    \end{align}
    for some $\gamma_{\lambda}\in X$. 
    Recall that $\phi$ satisfies the two-scale relation (3). Choosing
    \begin{align*}
      \gamma_{\lambda}(\xi)\coloneqq \text{const.}\, \gamma(\xi) S_{\lambda}(2\xi)\in X
    \end{align*}
    we thus conclude $\phi_{-1, \lambda}\in V_0$ and therefore $V_{-1}\subseteq V_{0}$.
\end{proof}

\begin{theorem}\label{thm:mra:schnittSkalenrme}
  Let $\phi\in L^2(\overline{\a_{+}}, \omega)$ satisfy condition (RB) and assume that the scale spaces are nested, i.e. $V_{j}\subseteq V_{j+1}$ for all $j\in \Z$. Suppose further that $|\hkl\phi|$ is continuous at $0$. Then $(V_{j})_{j\in\Z}$ is a radial MRA if and only if $\hkl\phi(0)\neq 0$. Moreover, if $\phi$ is an orthonormal scaling function then $|\hkl\phi(0)|=1$.
\end{theorem}

\begin{proof} The idea is the same as in the proof of \cite{roslerRauhut}, Thm. 4.9, but now the proof is much more involved, as sufficient knowledge of the generalized translation will be needed, which was simple and explicit in \cite{roslerRauhut}.  To start with,  we use \cref{thm:mra:orthogonalization} and obtain an orthonormal scaling function $\phi^{*}\in L^2(\overline{\a_{+}}, \omega)$ with the same scale spaces. In order to prove that  $\,\bigcap_{j=-\infty}^{\infty}V_{j}=\{0\}\,$ we have to show $\,\lim_{j\to-\infty}\|P_{j}f\|_{2} = 0$ for all $f\in L^2(\overline{\a_{+}}, \omega)$, where
    \begin{align*}
        P_{j}\colon L^2(\overline{\a_{+}}, \omega)\to V_{j},\quad f\mapsto \sum_{\lambda\in P_{+}}\langle f,\phi^{*}_{j,\lambda}\rangle\, \phi^{*}_{j,\lambda} 
    \end{align*}
    denotes the orthogonal projection and $\phi^{*}_{j,\lambda}$ is defined as in \eqref{eqn:mra:phiJLambda}. As continuous functions with compact support are dense in $L^2(\overline{\a_{+}}, \omega)$ we may assume that $f\in C_{c}(\overline{\a_{+}}).$ Denoting $K:= \text{supp}f$ and using Parseval's identity we obtain
    \begin{align*}
        \|P_{j}f\|_{2}^{2} &\ =\ \sum_{\lambda\in P_{+}} |\langle f, \phi^{*}_{j,\lambda}\rangle|^{2} \\
                           &\ \leq\ \|f\|_{2}^{2}\cdot \int_{K}\sum_{\lambda\in P_{+}}|\phi^{*}_{j,\lambda}(\xi)|^{2}\ \omega(\xi)d\xi\\ 
                           &\ =\,\text{const.} \int_{2^{j}\!K} \sum_{\lambda\in P_{+}} M_{\lambda}^{2}\, |T^{(\lambda)}\phi^{*}(\xi)|^{2}\ \omega(\xi)d\xi\\
                           &\ =\, \text{const.}\int_{2^{j}\!K\cap \a_{+}}\sum_{\lambda\in P_{+}} M_{\lambda}^{2}\ |(\delta_{\lambda+\rho}*\delta_{\xi})(\phi^{*})|^{2}\ \omega(\xi)d\xi,                       
    \end{align*}
   where above and in the sequel, const. denotes a varying positive constant depending on $f$ and $n$ only.  Observe that $\lambda+\rho\in \a_{+}$ for all $\lambda \in  P_{+}$ and therefore  \cref{prop:dichteTranslation} applies to $\delta_{\lambda+\rho} \ast \delta_\xi\,$ in the integral above.  We may further assume that $j\in \mathbb Z$ is negative and already so small that $\lambda+\rho + C(\xi) \subseteq \a_{+}$ for all $\lambda\in P_+.$ Hence,  by Lemma \ref{supp}, $\text{supp}(\delta_{\lambda+\rho}\ast\delta_\xi) \in \lambda+\rho + C(\xi)   
   $ for all $\lambda.$  
        In view of \cref{prop:dichteTranslation} we  obtain
   \begin{equation}\label{firstid} (\delta_{\lambda+\rho}*\delta_{\xi})(\phi^{*}) = \int_{\lambda+\rho+C(\xi)} \phi^{*}(x) \frac{\pi(\rho)\pi(x)}{\pi(\lambda+\rho)\pi(\xi)} \sum_{v, w\in \mathcal S_{n}}\sign(v)\sign(w)T(v\xi+w(\lambda+\rho)-x)\ dx.\end{equation}
     Here $dx$ denotes the Lebesgue volume in the affine hyperplane
    $\lambda^{1}+\xi^{1}+\R^{n}_{0}\ \subseteq\ \R^{n}\,$ which contains $\lambda+\rho+C(\xi),$ where still $\xi \in 2^jK\cap \a_{+}$ with  $j<0.$ Recall at this point  the decomposition $\mathbb R^n = \mathbb R_0^n \oplus \mathbb R\underline 1,$ with orthogonal projections $x^0$ of $x\in \mathbb R^n$ onto $\mathbb R_0^n$ and $x^1$ onto $\mathbb R\underline 1.$    
   Using Lemmata \ref{lem:eigenschaftenT} and \ref{lem:eigenschaftenOrdnung} we further see that for all 
   $x\in \lambda+\rho+C(\xi)$ and $v, w\in \mathcal S_{n}$, 
       \begin{align*}
        T(v\xi+w(\lambda+\rho)-x) \leq T(\lambda+\rho+\xi-x). 
    \end{align*}
    Moreover, as $\,x-(\lambda+\rho) \in C(\xi),$ we may estimate the argument of $T$ as  $\, \|x-(\lambda+\rho) -\xi\|_2 \leq 2 \|\xi\|_2\leq \sup_{x\in K}\!\|x\|_2\,.$ By its properties noted in Lemma \ref{lem:eigenschaftenT}, $T$ is therefore uniformly bounded on the domain of integration in \eqref{firstid}, i.e. with a bound independent of $\lambda$ and $\xi$. 
        By the Cauchy-Schwarz inequality and recalling that $\pi^2 = \omega,$  we  may therefore estimate
    \begin{align*}
        |(\delta_{\lambda+\rho}*\delta_{\xi})(\phi^{*})|^{2} \,
             &\leq\,\frac{\text{const.}}{\pi(\lambda+\rho)^{2}\omega(\xi)} \int_{\lambda+\rho+C(\xi)} |\phi^{*}(x)|^{2}\, \omega(x)\, dx, 
    \end{align*}
        where the constant is again independent of $\lambda$ and $\xi.$
    Now fix $R>0$ such that $ K\subseteq A+B\, $ with 
    $$ A= \{x=x^0\in \mathbb R_0^n \cap \overline{\a_{+}}: \|x\|_2\leq R\}, \quad B= \{x=x^1\in \mathbb R\underline 1 \cap  \overline{\a_{+}}: \|x\|_2\leq R\}$$ and recall that $M_{\lambda}=\frac{\pi(\lambda+\rho)}{\sqrt{n!}\prod_{k=1}^{n-1}k!}$.    
     Then we may continue our previous estimate as follows:
    \begin{align*}
        \|P_{j}f\|_{2}^{2} \,\leq \,\text{const.} \int_{2^{j}\!(A+B)}\sum_{\lambda\in P_{+}}\Bigl(\int_{\lambda+\rho+C(\xi)}|\phi^{*}(x)|^{2}\omega(x)dx\Bigr) d\xi.
    \end{align*}
        Noting that $\,C(\xi) = \xi^1 + C(\xi^0)$ with $\,C(\xi^0)\subseteq 2^j\!A, \,\xi^1 \in 2^j\!B,$ we conclude that
           \begin{align*}
        \|P_{j}f\|_{2}^{2} \leq\,& \text{const.}\int_{2^j\!A} \sum_{\lambda\in P_+}\Bigl(\int_{2^jB} \int_{\lambda+\rho+\xi^1+C(\xi^0)} |\phi^{*}(x)|^{2}\omega(x) dx\,d\xi^1\Bigr)d\xi^0\\
        \leq \,& \text{const.} \int_{2^j\!A}\sum_{\lambda\in P_+}\Bigl(\int_{\lambda+\rho+2^j(A+B)}|\phi^{*}(t)|^{2}\omega(t) dt\Bigr)d\xi^0.       \end{align*}
      Choosing $j\in \mathbb Z$ sufficiently small, we can achieve that the sets $\lambda+\rho+2^j(A+B)$ with $\lambda\in P_{+}$ are pairwise disjoint. This finally leads to the estimate
     $$  \|P_{j}f\|_{2}^{2} \leq\, \text{const.}\int_{2^j\!A}\Bigl(\int_{\overline{\a_{+}}}|\phi^{*}(t)|^{2}\omega(t) dt\Bigr)d\xi^0.     $$       
        The inner integral being finite, it follows that $\|P_jf\|_2 \to 0$ for $j\to -\infty$ as claimed. 
        This finishes the proof of condition (2) in the definition of a radial MRA, which is satisfied without requirements on $\mathcal H\phi.$ 

    It remains to analyze condition (3) concerning the density of $\,\bigcup_{j=-\infty}^{\infty}V_{j}\,$ in $L^2(\overline{\a_{+}}, \omega)$. Here we closely follow \cite{roslerRauhut}. Suppose first that $\hkl\phi(0)\neq 0$ and let $h\in \big(\bigcup_{j=-\infty}^{\infty}V_{j}\big)^{\perp}$, i.e. $P_{j}h = 0$ for all $j\in\Z$. Let $\varepsilon >0.$ By the Plancherel theorem for the Hankel transform, we can find  $f\in L^2(\overline{\a_{+}}, \omega)$ such that $\mathcal H f$  has compact support and $\|f-h\|_{2}\leq \varepsilon$. This implies
    \begin{align*}
        \|P_{j}f\|_{2} = \|P_{j}(f-h)\|_{2} \leq \varepsilon 
    \end{align*}
    for all $j\in\Z$. By the Riesz basis assumption on $\phi$, we further have
    \begin{align*}
        A\sum_{\lambda\in P_{+}} |\langle f, \phi_{j,\lambda}\rangle|^{2}\leq \|P_{j}f\|_{2}^{2}\leq B\sum_{\lambda\in P_{+}} |\langle f, \phi_{j,\lambda}\rangle|^{2}. 
    \end{align*}
    Suppose now that $\text{supp}\,\hkl f \subseteq K$ with some compact $K\subseteq\overline{\a_{+}}$. Then  in view of \eqref{eqn:radial:analogon43},     \begin{align*}
        \langle f, \phi_{j,\lambda}\rangle &= \langle \hkl f,\ \hkl \phi_{j,\lambda}\rangle \\
              &= \int_{K} \hkl f(\xi)\, \overline{\rho_{\lambda}^{j}\, \hkl\phi(2^{-j}\xi)}\ \omega(\xi)\ d\xi
    \end{align*}
    with the dilates $\,\rho^{j}_{\lambda}\coloneqq D_{2^{j}}(\frac{\Delta S_{\lambda}}{\pi}),\, \lambda\in P_{+}$. The weight $\omega$ being homogeneous,  they form an orthonormal basis of  the space $\,X_{j}\coloneqq L^{2}(2^{j}D, \omega)^{\mathcal S_{n}}$, where $$D= [0,2\pi[^n.$$   Assume $j$ is sufficiently large, so that  $2^{j}D\supseteq K$. Then 
    \begin{align*}
        \langle f, \phi_{j,\lambda}\rangle = \langle \hkl f\,\overline{\hkl\phi(2^{-j}\cdot)},\rho^{j}_{\lambda}\,\rangle_{X_{j}}.
    \end{align*}
    Thus using Parseval's equation for $X_{j}\cong L^{2}(2^{j}D\cap\overline{\a_{+}},\, \omega)$ we obtain 
    \begin{align}\label{eqn:mra:komischeGleichungImBeweisVereinigung}
        \sum_{\lambda\in P_{+}} |\langle f,\phi_{j,\lambda}\rangle|^{2} = \|\hkl f\, \hkl\phi(2^{-j}\cdot)\|_{X_{j}}^{2} = \int_{K} |\hkl f(\xi)|^{2} |\hkl\phi(2^{-j}\xi)|^{2}\ \omega(\xi)\, d\xi.
    \end{align}

    As we assumed that $|\hkl\phi|$ is continuous at $0$, the functions $|\hkl\phi(2^{-j}\cdot)|$ converge to the constant $|\hkl\phi(0)|>0$ uniformly on $K$ as $j\to+\infty$. Hence
    \begin{align*}
        \varepsilon \geq \limsup_{j\to\infty} \|P_{j}f\|_{2}\geq \sqrt{A}\,|\hkl\phi(0)|\,\|\hkl f\|_{2} \,\geq \,\sqrt{A}\,|\hkl\phi(0)|(\|h\|_{2}-\varepsilon). 
    \end{align*}
    But $\varepsilon > 0$ was arbitrary, so that $h=0$. This proves that  $\,\bigcup_{j=-\infty}^{\infty}V_{j}$ is dense in $L^2(\overline{\a_{+}}, \omega)$.

    Conversely, suppose that $\,\bigcup_{j=-\infty}^{\infty} V_{j}$ is dense in $L^2(\overline{\a_{+}}, \omega)$. Then
    \begin{align*}
        \lim_{j\to\infty} P_{j}f = f 
    \end{align*}
    for all $f\in L^2(\overline{\a_{+}}, \omega)$. If $\hkl f$ is compactly supported, the same calculation as before gives
    \begin{align*}
        \lim_{j\to\infty} \|P_{j}f\|_{2} \leq \sqrt{B}\,|\hkl\phi(0)|\,\|\hkl f\|_{2},
    \end{align*}
    which enforces $\hkl\phi(0)\neq 0$.  

    Finally, if the $\phi_{j, \lambda}, \lambda\in P_{+}$ are orthonormal,  we may choose $A=B=1$ and calculate
    \begin{align*}
        \|f\|_{2}=\lim_{j\to\infty} \|P_{j}f\|_{2}=|\hkl\phi(0)|\,\|\hkl f\|_{2} = |\hkl\phi(0)|\,\|f\|_{2},
    \end{align*}
    so that $|\hkl\phi(0)|=1$.
\end{proof}

Suppose that $\phi\in L^2(\overline{\a_{+}}, \omega)$ is an orthonormal scaling function of a radial MRA. Then according to  \cref{prop:mra:prop48analogon}, it satisfies the  two-scale relation
\begin{align*}
    \Delta(2\xi)\hkl\phi(2\xi) = \gamma(\xi)\Delta(\xi)\hkl\phi(\xi)
\end{align*}
with some $\gamma\in S$. 

\begin{definition} We introduce  the \emph{filter function}
\begin{align}\label{filter}
  G(\xi) \coloneqq \frac{\gamma(\xi)\Delta(\xi)}{\Delta(2\xi)}.
\end{align}
\end{definition}

Involving  $G,$ the above two-scale relation becomes  
$$\hkl\phi(2\xi)=G(\xi)\hkl\phi(\xi).$$ 
Note that $G$ is  $\mathcal S_{n}$-invariant and that $|G|$ is $I$-periodic.  Let us now also consider the finer lattice 

\begin{equation}\label{lattice}L\coloneqq \frac{1}{2}I = \pi\Z^{n} \supseteq I = 2\pi \mathbb Z^n.\end{equation}

\begin{lemma}\label{lem:mra:quadraturFilter}
    Suppose that $\phi\in L^2(\overline{\a_{+}}, \omega)$ is an orthonormal scaling function of a radial MRA. Then the associated filter function $G$ satisfies
    \begin{align*}
      \sum_{p\in L/I}|G(\xi+p)|^{2}\,= \,1 
    \end{align*}
    almost everywhere. As a consequence, $G$ is essentially bounded and contained in $L^{2}(D)$. 
\end{lemma}

\begin{proof}
    Using \cref{kor:mra:analogonKorr43} with $2\xi$ instead of $\xi$ and the two-scale relation, we get
    \begin{align*}
        1 &= \frac{1}{n!}\sum_{q\in I}|\hkl\phi(2\xi+q)|^{2}\,=\,
          \frac{1}{n!}\sum_{p\in L}|\hkl\phi(\xi+p)|^{2}|G(\xi+p)|^{2}\\
          &=\frac{1}{n!}\sum_{p\in L/I}|G(\xi+p)|^{2}\sum_{q\in I}|\hkl\phi((\xi+p)+q)|^{2}\,=\,\sum_{p\in L/I}|G(\xi+p)|^{2}.
    \end{align*}
\end{proof}


\section{Orthonormal Wavelets}\label{orthoWave}

Suppose we are given a radial MRA $(V_{j})_{j\in\Z}$ in $\text{Herm}(n)$ with orthonormal scaling function $\phi$. We define the wavelet space $W_{j}$ as the orthogonal complement of $V_{j}$ in $V_{j+1}$, i.e.
\begin{align*}
    V_{j+1} = V_{j}\oplus W_{j}. 
\end{align*}
As in classical multiresolution analysis (see e.g. \cite{Mal09}), the orthogonal projection of a radial function $f\in L^2(\text{Herm}(n))^{U(n)}$ onto $W_j$ gives the details of $f$ that appear at the resolution $2^{j+1}$ but disappear at the resolution $2^{j}.$
In this section, we will characterize orthonormal wavelets for the given radial MRA. That is,  we will give necessary and sufficient conditions for translations and dilations of functions $\psi^{1},\ldots, \psi^{r}\in L^2(\overline{\a_{+}}, \omega)$ in order to constitute an orthonormal basis of the wavelet space $W_{0}$. As 
$W_{j}=D_{2^{-j}}W_{0}\,,$ we obtain orthonormal bases of all spaces $W_j$ by dilation, and thus finally an orthonormal wavelet basis of
\begin{align*}
     L^2(\overline{\a_{+}}, \omega) = \mathop{\widehat\bigoplus}_{j\in\Z} W_{j}. 
\end{align*}

\begin{proposition}\label{prop:mra:onsCondition}
    For functions $\psi^{1},\ldots, \psi^{r-1}\in L^2(\overline{\a_{+}}, \omega)$ the following are equivalent:
    \begin{enumerate}
        \item $\{M_{\lambda}T^{(\lambda)}\psi^{i}\ :\ 1\leq i\leq r-1,\ \lambda\in P_{+}\}$ is an orthonormal system.
        \item The sum $$ P_{i, j}(\xi)\coloneqq \frac{1}{n!}\sum_{q\in I}\hkl\psi^{i}(\xi+q)\overline{\hkl\psi^{j}(\xi+q)}\,$$ is absolutely convergent with $P_{i,j}(\xi) = \delta_{i,j}$ for almost all $\xi\in\R^{n}.$
    \end{enumerate}
\end{proposition}
\begin{proof}
  Since the case $i=j$ is covered by \cref{kor:mra:analogonKorr43}, we may assume that $i\neq j$. Then the Plancherel theorem for the Hankel transform implies that 
    \begin{align*}
    \big\langle M_{\lambda}T^{(\lambda)}\psi^{i},\ M_{\mu}T^{(\mu)}\psi^{j}\big\rangle &= \Big\langle \frac{\Delta S_{\lambda}}{\pi}\hkl\psi^{i},\ \frac{\Delta S_{\mu}}{\pi}\hkl\psi^{j}\Big\rangle \\
  &= \int_{\overline{\a_{+}}} S_{\lambda}(\xi)\overline{S_{\mu}(\xi)}\hkl\psi^{i}(\xi)\overline{\hkl\psi^{j}(\xi)}\ |\Delta(\xi)|^{2} d\xi\\
    &= \frac{1}{n!}\int_{\R^{n}}S_{\lambda}(\xi)\overline{S_{\mu}(\xi)}\hkl\psi^{i}(\xi)\overline{\hkl\psi^{j}(\xi)}\ |\Delta(\xi)|^{2} d\xi\\
                 &=\int_{\mathbb T^n} S_{\lambda}(\xi)\overline{S_{\mu}(\xi)} P_{i,j}(\xi)\ |\Delta(\xi)|^{2} d\xi.
    \end{align*}
    This immediately gives $(2)\Rightarrow (1)$, since $(S_\lambda)_{\lambda \in P_+}$ is an orthonormal basis of $X=L^2(\mathbb T^n, |\Delta|^2)^{\mathcal S_n}.$  For the converse implication, note first that the series defining $P_{i,j}$ is absolutely convergent with $|P_{i,j}|\leq 1$ a.e., as a consequence of \cref{kor:mra:analogonKorr43}    and    the Cauchy-Schwarz inequality in $\ell^2(I).$ Moreover, 
    $P_{i,j}$ is $I$-periodic and $\mathcal S_n$-invariant (recall that the Hankel transform $\hkl\psi$ is $\mathcal S_{n}$-invariant). Note that $P_{i,j} = \overline{P_{j,i}}.$ By our assumption and the above calculation,
    \begin{align*}
      \big\langle S_{\lambda}P_{j,i}\,, S_{\mu}\rangle_{X}  = 0
    \end{align*}
   for all $\lambda, \mu\in P_+.$ Thus $S_{\lambda}P_{j,i}=0$ a.e. for every $\lambda$ and hence $\,\langle S_\lambda, P_{i,j}\rangle_X = 0\,$ for every $\lambda. $ 
   This implies that $P_{i,j} = 0$ a.e.
\end{proof}

Maintaining the above setting, we shall now characterize the space $W_{-1}\subseteq V_{0}$. 

\begin{proposition}\label{prop:mra:characterizeW1}
    For $f\in L^2(\overline{\a_{+}}, \omega)$ the following statements are equivalent:
    \begin{enumerate}
        \item $f\in W_{-1}$.
        \item  There is an element $\beta\in X$ such that $\, \displaystyle \hkl f(\xi)=\frac{\Delta(\xi)\beta(\xi)}{\pi(\xi)}\hkl\phi(\xi)\,$ and
            \begin{align*}
                (\beta(\xi+p) \delta(\xi+p))_{p\in L/I} \perp (\gamma(\xi+p)\delta(\xi+p))_{p\in L/I} 
            \end{align*}
            almost everywhere as vectors in $\C^{r},$ where  $L$ is the lattice \eqref{lattice},
            $$r=|L/I| = 2^{n}, \quad \delta(x)\coloneqq\frac{\Delta(x)}{\Delta(2x)}$$ 
            and $\gamma\in X$ is the function from the two-scale relation for $\phi$ in \cref{prop:mra:prop48analogon}.
    \end{enumerate}
\end{proposition}
\begin{proof}
    By definition $V_{0}=W_{-1}\oplus V_{-1}$ and $V_{-1}=D_{2}V_{0}$, thus $h\in V_{-1}$ if and only if $D_{1/2}h\in V_{0}$. In view of \cref{lem:mra:analogonLemma45} this is equivalent to
    \begin{align*}
        \hkl(D_{1/2}h)(\xi) = \frac{\Delta(\xi)\eta(\xi)}{\pi(\xi)}\hkl\phi(\xi) 
    \end{align*}
    with some $\eta\in X$. From \cref{eqn:hklVsDilation} we obtain
    \begin{align*}
        2^{-\frac{m}{2}}\hkl h(\tfrac{\xi}{2}) = (D_{2}\hkl h)(\xi) = \hkl(D_{1/2}h)(\xi)
    \end{align*}
    and thus
    \begin{align*}
      h\in V_{-1} \quad\Longleftrightarrow\quad \hkl h(\xi) = \frac{2^{\frac{m}{2}}\Delta(2\xi)\eta(2\xi)}{\pi(2\xi)}\hkl\phi(2\xi) = \frac{\beta(2\xi)}{\pi(\xi)}\gamma(\xi)\Delta(\xi)\hkl\phi(\xi)
    \end{align*}
    for some $\beta\in X.$ Furthermore, $f\in V_{0}$ if and only if
    \begin{align*}
        \hkl f(\xi)=\frac{\Delta(\xi)\widetilde\beta(\xi)}{\pi(\xi)}\hkl\phi(\xi) 
    \end{align*}
    for some $\widetilde\beta\in X$. We denote $$ D' \coloneqq \frac{1}{2}D = [0,\pi[^{n},$$ which is a fundamental domain of the smaller torus $\R^{n}/L$. As $W_{-1}\perp V_{-1}$ and $\hkl$ is unitary, we have $f\in W_{-1}$ if and only if $\langle \hkl f, \hkl h\rangle = 0$ for all $h\in V_{-1}$, i.e. if and only if for all $\beta\in S$ we have
    \begin{align*}
      0 &= \int_{\overline{\a_{+}}} \frac{\widetilde\beta(\xi)}{\pi(\xi)}\,\frac{\overline{\beta(2\xi)}}{\overline{\pi(\xi)}}\overline{\gamma(\xi)} |\hkl\phi(\xi)|^{2}\, |\Delta(\xi)|^{2} \omega(\xi)d\xi\\
        &= \frac{1}{n!} \int_{\R^{n}}\widetilde\beta(\xi)\overline{\beta(2\xi)}\overline{\gamma(\xi)}|\hkl\phi(\xi)|^{2}\, |\Delta(\xi)|^{2} d\xi\\
        &= \frac{1}{n!} \sum_{p\in L/I}\sum_{q\in I} \int_{D'} \overline{\beta(2\xi)}\,\widetilde\beta(\xi+p)\overline{\gamma(\xi+p)}\, |\hkl\phi((\xi+p)+q)|^{2}\, |\Delta(\xi+p)|^{2} d\xi\\
        &= \int_{D'} \overline{\beta(2\xi)}\sum_{p\in L/I}\widetilde\beta(\xi+p)\Delta(\xi+p)\overline{\gamma(\xi+p)\Delta(\xi+p)}\, P_\phi(\xi + p)  d\xi\\
        &= \int_{D'} \overline{\beta(2\xi)}\sum_{p\in L/I}\widetilde\beta(\xi+p)\delta(\xi+p)\overline{\gamma(\xi+p)\delta(\xi+p)}\ |\Delta(2\xi)|^{2} d\xi.
    \end{align*}
    Here it was used that $P_\phi = 1$ a.e. and that $\beta(2\,\cdot)$ is $L$-periodic while $\widetilde \beta, \gamma$ and $|\Delta|$ are $I$-periodic.  Note that also $|\delta|^2$ is $I$-periodic and $\mathcal S_{n}$-invariant, and therefore the finite sum
    \begin{align*}
        \sum_{p\in L/I} \widetilde\beta(\xi+p)\delta(\xi+p)\overline{\gamma(\xi+p)\delta(\xi+p)} 
    \end{align*}
    is $\mathcal S_{n}$-invariant and $L$-periodic. Since the  $\beta(2\,\cdot),\ \beta\in S$ exhaust the space $L^{2}(D', |\Delta(2\,\cdot)|^{2})^{\mathcal S_{n}}$, we conclude that
$$
      \sum_{p\in L/I} \widetilde\beta(\xi+p)\delta(\xi+p)\overline{\gamma(\xi+p)\delta(\xi+p)} = 0 \quad \text{ a.e.}
 $$
\end{proof}

\begin{theorem}\label{thm:mra:orthogonaleWavelets}
    For elements $\psi^{1},\ldots,\psi^{r-1}\in L^2(\overline{\a_{+}}, \omega)$ the following are equivalent:
    \begin{enumerate}
        \item The set
            \begin{align*}
                \{M_{\lambda}T^{(\lambda)} \psi^{i}\ :\ 1\leq i\leq r-1,\ \lambda\in P_{+}\} 
            \end{align*}
            is an orthonormal basis  of $W_{0}$.
        \item The number $r$ is given by  $\, r= |L/I| = 2^{n},$ and
            \begin{align*}
              \hkl\psi^{i}(2\xi) = \beta^{i}(\xi) \delta(\xi)\hkl\phi(\xi) 
            \end{align*}
            for certain $\beta^{i}\in X$ such that the matrix
            \begin{align*}
                (\beta^{i}(\xi+p)\delta(\xi+p))_{0\leq i\leq r-1,\, p\in L/I}\in \C^{r\times r} 
            \end{align*}
            is unitary almost everywhere. Here $\beta^{0}\coloneqq \gamma \in X$ denotes the function from the two-scale relation of $\phi$ in \cref{prop:mra:prop48analogon}.
    \end{enumerate}
\end{theorem}

\begin{proof}
    Statement $(1)$ is equivalent to the statement that the $M_{\lambda}D_{2}T^{(\lambda)}\psi^{i}$ are an orthonormal basis of $D_{2}W_{0} = W_{-1}$. Using \cref{prop:mra:characterizeW1} we see $M_{0}D_{2}T^{(0)}\psi^{i}\in W_{-1}$ if and only if
    \begin{align*}
        \hkl(M_{0}D_{2}T^{(0)}\psi^{i})(\xi) = \frac{\Delta(\xi)\widetilde\beta^{i}(\xi)}{\pi(\xi)}\hkl\phi(\xi) 
    \end{align*}
    with certain $\widetilde\beta^{i}\in X$ such that $(\widetilde\beta^{i}(\xi+p)\delta(\xi+p))_{p\in L/I}\perp (\gamma(\xi+p)\delta(\xi+p))_{p\in L/I}$ almost everywhere. Recalling formula \eqref{eqn:radial:analogon43},    we calculate the left hand side of the previous equation as
    \begin{align*}
      \hkl(M_{0}D_{2}T^{(0)}\psi^{i})(\xi) &= D_{2^{-1}}\bigg(\frac{\Delta S_{0}}{\pi}\,\hkl\psi^{i}\bigg)(\xi) \\
        &= \text{const.}\, \frac{\Delta(2\xi)}{\pi(\xi)}\,\hkl\psi^{i}(2\xi).
    \end{align*}
    We conclude that $\,M_{0}D_{2}T^{(0)}\psi^{i}\in W_{-1}$ if and only if
    \begin{align}\label{eqn:mra:twoscaleWaveletRelation}
        \hkl\psi^{i}(2\xi) = \beta^{i}(\xi)\delta(\xi)\hkl\phi(\xi)
    \end{align}
    for some $\beta^{i}\in X$ such that $(\beta^{i}(\xi+p)\delta(\xi+p))_{p\in L/I} \perp (\gamma(\xi+p)\delta(\xi+p))_{p\in L/I}$ almost everywhere. By \cref{prop:mra:onsCondition} we see that $\{M_{\lambda}T^{(\lambda)}\psi^{i} : 1\leq i\leq r-1,\ \lambda\in P_{+}\}$ is an orthonormal system if and only if
    \begin{align*}
        \delta_{ij} &= \frac{1}{n!}\sum_{q\in I} \hkl\psi^{i}(2\xi+q)\overline{\hkl\psi^{j}(2\xi+q)}\\ 
                    &=\frac{1}{n!}\sum_{p\in L}\hkl\psi^{i}(2(\xi+p))\overline{\hkl\psi^{j}(2(\xi+p))}\\
                    &=\frac{1}{n!}\sum_{p\in L/I}\sum_{q\in I}\beta^{i}(\xi+p+q)\overline{\beta^{j}(\xi+p+q)}|\delta(\xi+q+p)|^{2}\,|\hkl\phi(\xi+p+q)|^{2}\\
                    &=\sum_{p\in L/I} \beta^{i}(\xi+p)\overline{\beta^{j}(\xi+p)}|\delta(\xi+p)|^{2}\,\frac{1}{n!}\sum_{q\in I}|\hkl\phi((\xi+p)+q)|^{2}\\
                    &=\sum_{p\in L/I}\beta^{i}(\xi+p)\overline{\beta^{j}(\xi+p)}|\delta(\xi+p)|^{2}
    \end{align*}
    almost everywhere. This is equivalent to the condition that the set 
    $$M\coloneqq\{(\beta^{i}(\xi+p)\delta(\xi+p))_{p\in L/I} : 1\leq i\leq r-1\}$$ 
    is an orthonormal system in $\C^{|L/I|}$ a.e. Furthermore, the vector $(\beta^{0}(\xi+p)\delta(\xi+p))_{p\in L/I}$ is almost everywhere  normalized  in $\C^{|L/I|}$ according to \cref{lem:mra:quadraturFilter}, and is orthogonal to $M$ according to \cref{prop:mra:characterizeW1}. We conclude that $r\leq |L/I|$ is a necessary condition by dimensionality. Moreover, if the $\psi^{i}$ satisfy  \eqref{eqn:mra:twoscaleWaveletRelation}, then a short calculation gives
    \begin{align*}
      \hkl(M_{\lambda}D_{2}T^{(\lambda)}\psi^{i})(\xi) = \text{const.}\cdot \frac{\Delta(\xi)S_{\lambda}(2\xi)\beta^{i}(\xi)}{\pi(\xi)}\hkl\phi(\xi).
    \end{align*}
    So indeed, $M_{\lambda}D_{2}T^{(\lambda)}\psi^{i}\in V_{0}\, $ for all $\lambda\in P_{+}$ by \cref{lem:mra:analogonLemma45}. As $S_{\lambda}(2\,\cdot)$ is $L$-periodic, we have 
    \begin{align*}
      \bigl(S_{\lambda}(2(\xi+p))\beta^{i}(\xi+p)\delta(\xi+p)\bigr)_{p\in L/I} \in \C\, \bigl(\beta^{i}(\xi+p)\delta(\xi+p)\bigr)_{p\in L/I}  \end{align*} 
    and thus $M_{\lambda}D_{2}T^{(\lambda)}\psi^{i}\in W_{-1}\subseteq V_{0}$ by \cref{prop:mra:characterizeW1}. We conclude that $r=|L/I|$ is also  a necessary condition.
\end{proof}

\begin{corollary}
    Every radial MRA $(V_j)_{j\in \mathbb Z}$ with orthonormal scaling function $\phi$ admits an orthonormal wavelet basis consisting of $r-1$ wavelets, where $r=|L/I| = 2^{n}$, i.e. there are functions $\psi^{1},\ldots,\psi^{r-1}\in L^2(\overline{\a_{+}}, \omega)$ such that
    \begin{align*}
        \{M_{\lambda}D_{2^{-j}}T^{(\lambda)}\psi^{i}\ :\ 1\leq i\leq r-1,\ \lambda\in P_{+}\} 
    \end{align*}
    is an orthonormal basis of the complementary space $W_{j}$.
\end{corollary}

\begin{proof} Recall that $\,\Delta(x)=\alpha(x)A_{\delta}(e^{ix})$
with the phase factor $\,\alpha(\xi)=e^{-i\frac{n-1}{2}\langle \xi, \underline 1\rangle }.$ Hence $\alpha\Delta$ is $I$-periodic. 
  We now put $\,\eta^{0}(\xi)=\alpha(-\xi)G(\xi)$ with the filter function $G$ associated to $\phi$ according to formula \eqref{filter}.  We next choose functions $\eta^{i}\in L^{2}(D\cap\overline{\a_{+}}), 1\leq i\leq r-1$, in such a way that the matrix
    \begin{align*}
        (\eta^{i}(\xi+p))_{0\leq i\leq r-1,\, p\in L/I}\in\C^{r\times r} 
    \end{align*}
    is unitary for almost all $\xi\in D\cap\overline{\a_{+}}\,$; here again $D = [0, 2\pi[^n.$ 
    This amounts to constructing a unitary matrix with a given first row in a measurable way.  
    We then extend the $\eta^{i}$ to all of $D$ via $\eta^{i}(wx)\coloneqq \eta^{i}(x)$ for $w\in\mathcal S_{n}$ and then to $\R^{n}$ in an $I$-periodic fashion. 
    We thus obtain functions $\beta^{i}\coloneqq \alpha\delta^{-1}\eta^{i}\in X$ such that the matrix $A=(A_{ip})_{0\leq i\leq r-1,\, p\in L/I}$ with entries
    \begin{align*}
      A_{ip} = \beta^{i}(\xi+p)\delta(\xi+p) = \alpha(\xi)\eta^i(\xi+p)
    \end{align*}
    is almost everywhere unitary. In particular, $\beta^{0}\delta = G$.
    We then define $\psi^{i}  \in  L^2(\overline{\a_{+}}, \omega)$ via its Hankel transform by 
    \begin{align*}
      \hkl\psi^{i}(2\xi) \coloneqq \beta^{i}(\xi)\delta(\xi)\hkl\phi(\xi). 
    \end{align*}
\end{proof}


\section{Construction of radial scaling functions}\label{scaling}

It remains to discuss how a radial MRA can be actually obtained,  i.e. how  candidates for radial scaling functions can be found. 
As our lattice is  $\,I=2\pi\Z^{n}$, we can tile $\R^{n}$ with copies of $D=[0, 2\pi[^{n}$ along periods of $e^{i\langle \xi, \cdot\rangle}$, which allows us to interlock \cref{prop:mra:analogonProp42} with its classical analogue for the Euclidean Fourier transform $\widehat  f(\xi) = (2\pi)^{-n/2}\int_{\mathbb R^n} e^{-i\langle \xi,x\rangle}dx $ on $L^2(\mathbb R^n).$ 

\begin{theorem}\label{thm:mra:konstruktionSkalenfktHerm}
  Suppose $\phi_{\a}\in L^{2}(\R^{n})$ is a classical scaling function for a dyadic MRA in $L^2(\mathbb R^n)$  which is $\mathcal S_{n}$-invariant and such that its classical Fourier transform $\widehat\phi_{\a}$ is continuous at $0$ and satisfies $\widehat\phi_{\a}\in L^2(\overline{\a_{+}}, \omega)$. Then
    \begin{align}\label{eqn:mra:konstruktionSkalenfunktion}
      \hkl\phi(\xi)\coloneqq (2\pi)^{\frac{n}{2}} e^{-i\frac{n-1}{2}\langle \xi, \underline 1\rangle} \widehat\phi_{\a}(\xi) 
    \end{align}
    defines a radial scaling function $\phi\in L^2(\overline{\a_{+}}, \omega)$, i.e. a scaling function for a radial MRA in $\text{Herm(n)}$.  \newline\hspace*{6mm}Conversely, if $\phi\in L^2(\overline{\a_{+}}, \omega)$ is a radial scaling function such that $\hkl\phi\in L^{2}(\R^{n})$ and $\hkl\phi$ is continuous at $0$, then the function $\phi_{\a}$ defined by \cref{eqn:mra:konstruktionSkalenfunktion} is a classical scaling function on $\R^{n}$ which is $\mathcal S_{n}$-invariant.
    \newline\hspace*{6mm}Moreover, $\phi$ is an orthonormal scaling function if and only if $\phi_{\a}$ is an orthonormal classical scaling function.
\end{theorem}

\begin{proof}
  Suppose $\phi_{\a}$ is a classical scaling function which is $\mathcal S_{n}$-invariant and such that $\widehat\phi_{\a}$ is continuous at $0$. Note that the $\mathcal S_{n}$-invariance implies that $\widehat\phi_{\a}$ is  $\mathcal S_{n}$-invariant as well, and thus definition \eqref{eqn:mra:konstruktionSkalenfunktion} is meaningful. 
    According to Propos. 5.7. in \cite{wojtaszczyk}  there are constants $0<A\leq B<\infty$ such that
    \begin{align}\label{schachtel}
        \frac{A}{(2\pi)^{n}}\leq \sum_{l\in\Z^{n}} |\widehat \phi_{a}(\xi+2\pi l)|^{2} \leq \frac{B}{(2\pi)^{n}} \quad \text{a.e.}.
    \end{align}
    Moreover,  $\phi_{\a}$ is an orthonormal (classical) scaling function if and only if $A=B=1$. Since $I=2\pi \Z^{n}$, formula \eqref{schachtel} can be written as 
    \begin{align*}
        A \leq \sum_{q\in I} |\hkl \phi(\xi+q)|^{2} \leq B \quad \text{a.e.}.
    \end{align*}
    So we can invoke \cref{prop:mra:analogonProp42} and \cref{kor:mra:analogonKorr43} to see that $B_{\phi}\coloneqq\{M_{\lambda}T^{(\lambda)}\phi :\lambda\in P_{+}\}$ forms a Riesz basis of $V_{0}\coloneqq \overline{\text{span}\, B_{\phi}}$, and that this basis is orthonormal if and only if $\phi_{\a}$ is orthonormal. Further, by Lemma 5.8 in \cite{wojtaszczyk} there exists  a $2\pi\Z^{n}$-periodic function $m$ on $\R^{n}$ with $m|_{D}\in L^{2}(D)$ and such that $\widehat\phi_{\a}(2\xi) = m(\xi)\widehat\phi_{\a}(\xi)$. Since $\widehat\phi_{\a}$ was assumed to be $\mathcal S_{n}$-invariant, $m$ has to be $\mathcal S_{n}$-invariant as well.
    We again introduce the phase factor $\alpha(\xi)=e^{-i\frac{n-1}{2}\langle \xi,\underline{1}\rangle}$ and define $\,\gamma\coloneqq\alpha\delta^{-1}m\in S.\,$ We obtain
    \begin{align*}
      \Delta(2\xi) \hkl\phi(2\xi) &= \Delta(2\xi) \alpha(2\xi)(2\pi)^{\frac n 2} m(\xi) \widehat\phi_{\a}(\xi) \\
                                  &= \frac{\Delta(2\xi)\alpha(2\xi)\alpha(\xi)^{-1}m(\xi)}{\Delta(\xi)}\Delta(\xi)\hkl\phi(\xi)\\ 
                                  &= \gamma(\xi) \Delta(\xi)\hkl\phi(\xi),
    \end{align*}
    which is just the two-scale relation from \cref{prop:mra:prop48analogon}. As $\widehat\phi_{\a}$ is continuous at $0$, we get $\widehat\phi_{a}(0)\neq 0$ and thus $\hkl \phi(0) \neq 0$, see \cite{daubechies1992ten}, Remark 3 on p. 144. Indeed, we have seen this reasoning already in the proof of \cref{thm:mra:schnittSkalenrme}. Hence, the conditions of \cref{thm:mra:schnittSkalenrme} are satisfied and we obtain that $\phi$ is indeed a radial scaling function.

    Conversely, suppose that $\hkl\phi\in L^{2}(\R^{n})$.  We proceed as before and note that
    \begin{align*}
      m(\xi) = \gamma(\xi)\delta(\xi)\alpha(\xi)^{-1} = \alpha(-\xi)G(\xi).
    \end{align*}
    By \cref{lem:mra:quadraturFilter}, this shows that  
$m\in L^{2}(D).$
\end{proof}

\begin{example} Es an example, we consider the radial analogue of the Shannon wavelets. Again, we work with  the phase factor $\alpha(\xi)\coloneqq e^{-i\frac{n-1}{2}\langle \xi, \underline 1\rangle}$.
Consider $\phi_{\a}$ defined via its classical Fourier transform $\widehat\phi_{\a}=(2n)^{-n/2}\chi_{[-\pi,\pi]^{n}}$. Then by \cite{wojtaszczyk}, Thm. 2.13, Prop. 5.7 and Lem. 5.8, 
$\phi_{\a}$ is a classical orthonormal scaling function  which satisfies the conditions of our \cref{thm:mra:konstruktionSkalenfktHerm}.  Therefore  $\,\hkl\phi(\xi) \coloneqq \alpha(\xi)\chi_{[-\pi,\pi]^{n}}(\xi)\,$ defines an orthonormal radial scaling function $\phi$. Note that
\begin{align}\label{eqn:scaling:scalingTwoScale}
    \hkl\phi(2\xi) = \alpha(\xi)\chi_{Q}(\xi)\hkl\phi(\xi) = \alpha(2\xi)\chi_{Q\,\cap\,[-\pi,\pi]^{n}}(\xi)
\end{align}
with the union of cubes
\begin{align*}
    Q \coloneqq Q^{0} \coloneqq \bigcup_{l\in\Z^{n}}(2\pi l + [-\pi/2, \pi/2]^{n}). 
\end{align*}
Now define functions $\,\beta^{i} = (\alpha\delta)^{-1}\chi_{Q^{i}}\in X$ with
\begin{align*}
    Q^{i} = q^{i}+Q, 
\end{align*}
where $q^{i}$ runs through the non-trivial representatives of $L/I$, i.e. through the set $\{0, \pi\}^{n}\setminus \{(0,\ldots, 0)\}$. Then in view of 
\cref{thm:mra:orthogonaleWavelets}, one obtains an orthonormal radial wavelet basis $(\psi^{i})$ for $\text{Herm}(n)$ by defining
\begin{align}\label{eqn:scaling:waveletTwoScale}
    \hkl\psi^{i}(2\xi) \coloneqq \beta^{i}(\xi)\delta(\xi)\hkl\phi(\xi)=\chi_{Q^{i}\cap [-\pi, \pi]^{n}}(\xi).
\end{align}


\begin{figure}[H]
    \centering
    \begin{minipage}{0.18\textwidth}
        \scalebox{0.6}{
        \begin{tikzpicture}
            \draw[step=0.5, lightgray] (0.5,0.5) grid +(2.0, 2.0);
            \draw (0.5,0.5) rectangle +(2.0,2.0);
            \fill (1, 1) rectangle +(1, 1);
            \node[align=center] at (1.5, 0) {$Q^{0}\cap [-\pi,\pi]^{2}=[-\pi/2, \pi/2]^{2}$};
        \end{tikzpicture}
        }
    \end{minipage}
    \hfill
    \begin{minipage}{0.18\textwidth}
        \scalebox{0.6}{
        \begin{tikzpicture}
            \draw[step=0.5, lightgray] (0.5,0.5) grid +(2.0, 2.0);
            \draw (0.5,0.5) rectangle +(2.0,2.0);
            \fill (1, 2) rectangle +(1, 0.5);
            \fill (1, 0.5) rectangle +(1, 0.5);
            \node[align=center] at (1.5, -0) {$Q^{1}\cap [-\pi,\pi]^{2}$ with $q^{1}=(0, \pi)$};
        \end{tikzpicture}
        }
    \end{minipage}
    \hfill
    \begin{minipage}{0.18\textwidth}
        \scalebox{0.6}{
        \begin{tikzpicture}
            \draw[step=0.5, lightgray] (0.5,0.5) grid +(2.0, 2.0);
            \draw (0.5,0.5) rectangle +(2.0,2.0);
            \fill (0.5, 1) rectangle +(0.5, 1);
            \fill (2, 1) rectangle +(0.5, 1);
            \node[align=center] at (1.5, -0) {$Q^{2}\cap [-\pi,\pi]^{2}$ with $q^{2}=(\pi, 0)$};
        \end{tikzpicture}
        }
    \end{minipage}
    \hfill
    \begin{minipage}{0.18\textwidth}
        \scalebox{0.6}{
        \begin{tikzpicture}
            \draw[step=0.5, lightgray] (0.5,0.5) grid +(2.0, 2.0);
            \draw (0.5,0.5) rectangle +(2.0,2.0);
            \fill (0.5, 0.5) rectangle +(0.5, 0.5);
            \fill (2, 0.5) rectangle +(0.5, 0.5);
            \fill (0.5, 2) rectangle +(0.5, 0.5);
            \fill (2, 2) rectangle +(0.5, 0.5);
            \node[align=center] at (1.5, -0) {$Q^{3}\cap [-\pi,\pi]^{2}$ with $q^{3}=(\pi, \pi)$};
        \end{tikzpicture}
        }
    \end{minipage}

    \caption{The radial Shannon wavelets in the rank $n=2$ case. The sets $Q^{i}\cap[-\pi, \pi]^{2}$ correspond to the scaling function $\phi$ ($i=0$) and the three wavelets $\psi^{i} (i=1,\ldots, 3)$ via \cref{eqn:scaling:scalingTwoScale} and \cref{eqn:scaling:waveletTwoScale}.}
\end{figure}
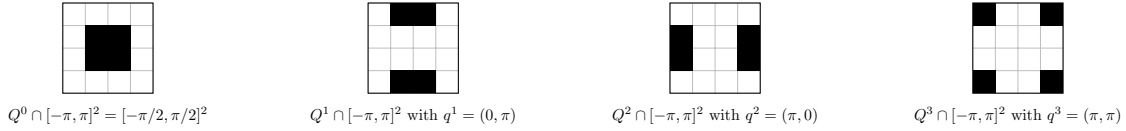
\end{example}


\section{Appendix: Generalizations}\label{appendix}

Many of the previous results can be generalized in terms of Lie theory. Instead of the action of $U(n)$ on $\text{Herm}(n)$ by conjugation we may consider an arbitrary compact connected Lie group $K$ which acts on $\mathfrak{p}=i\text{Lie}(K)\subseteq\text{Lie}(K)_{\C}$ via the adjoint representation as a group of orthogonal transformations with respect to some $K$-invariant inner product on $\mathfrak{p}$.   First, we replace $\R^{n}$ by a maximal abelian subspace $\a = i\text{Lie}(T)\subseteq\mathfrak{p}$ coming from a maximal torus $T\subseteq K$ and $\R^{n}_{0}$ by $\a_{0}$, the span of the roots of $K$.
The Weyl integration formula holds in this generality (cf. \cite{gangolliVaradarajan}, Eq. (2.4.22)). Further, $\overline{\a_{+}}$ is now the positive Weyl chamber associated with a set of simple positive roots $\alpha_{1},\ldots, \alpha_{\dim\a_{0}}$. As usual, we identify $\a$ with $\a^{*}$ using a $K$-invariant inner product on $\mathfrak{p}$.
The associated Weyl group $W$ replaces the symmetric group $\mathcal S_{n}$ and $\pi$ becomes the usual alternating polynomial $\pi=\prod_{i=1}^{q}\alpha_{i}$  with respect to the positive roots $\alpha_{1},\ldots,\alpha_{q}$. Again, we define a generalized translation by averaging the translation along adjoint orbits of $K$ in $\mathfrak{p}.$ 

The Weyl character formula yields a natural choice for the trigonometric polynomials $S_{\lambda}$, which are indexed by the set of dominant weights ${P}_{+}$ of $K$ with respect to $T$. On the other hand the Harish-Chandra-Integral (\cite{harish1957differential}) allows us to express the spherical functions $J(\,.\,,iy)$ of $K\ltimes \mathfrak{p}/K$ in a similar fashion. Thus interlocking these two allows for nice properties of the $S_{\lambda}$ with respect to the generalized translation.
More precisely, we obtain
\begin{align}\label{S_Bessel_Reduktiv}
  S_{\lambda}(x) = \frac{1}{i^{q}\sqrt{|W|}\Delta(x)} \sum_{w\in W}\sign(w)e^{i(\lambda+\rho)(wx)} = M_{\lambda} \frac{\pi(x)}{\Delta(x)} J(x, i(\lambda+\rho))
\end{align}
as a generalization of \cref{S_Bessel}.
Here, $s_{\lambda}\coloneqq i^q\sqrt{|W|}S_{\lambda}$ is the trigonometric polynomial associated to the highest weight $\lambda$ in the Weyl character formula (\cite{knapp}, Thm. 5.113), $\Delta$ denotes the Weyl denominator and $M_{\lambda}=\text{const.}\cdot \pi(\lambda+\rho)$ is a suitable normalization constant determined by the Harish-Chandra-Integral (see for example \cite{mcswiggen}, Eq. 3.3).
Again, the Peter-Weyl theorem and the Schur orthogonality relations ensure that we obtain an orthonormal basis. 

Using a structure theorem for compact connected Lie groups (\cite{procesi}, Ch. 10, \S 7.2, Thm. 4), we can decompose $K=(L\times H)/C$ with $L$ compact, connected and semisimple, a torus $H$ and a finite subgroup $C$ of the center of $L\times H$ intersecting $H$ trivially. This allows us to reduce to the semisimple case as in the proof of \cref{prop:dichteTranslation} and apply the results of \cite{graczykSawyer_complex} to find the density of the translation. Note that $\a_{0}$ is the orthogonal complement of $i\text{Lie}(H)$ in $\a$ as $L, L\times H$ and $K$ all have the same roots.
We now take $I=i\ker(\exp\colon\text{Lie}(T)\to T)$ as the integral lattice of $T$ and replace $D$ by a fundamental domain of the torus $\a/I$. After adjusting these notations, the arguments up until \cref{thm:mra:orthogonaleWavelets} will still work. A small caveat is to see that the Weyl group acts on the integral lattice in order to see for example the $\mathcal S_n$-invariance of $P_{\phi}$ in \cref{prop:mra:analogonProp42}. However, employing the analytic Weyl group $W\cong N_{K}(T)/T$, we immediatly verify 
\begin{align*}
    \exp(\text{Ad}(k)x)=k\exp(x)k^{-1}=e
\end{align*}
for $w=\text{Ad}(k)\in W$ and $x\in\ker(\exp\text{Lie}(T)\to T)$.
After \cref{orthoWave}, however, the geometry will not be as compatible as in the case $K=U(n)$. For example, it is not clear how to obtain radial scaling functions as the lattice $I$ will generally not behave well with the periodicy of $e^{i\langle \xi,\cdot\rangle }$, which was crucial in the proof of \cref{thm:mra:konstruktionSkalenfktHerm}. 

In the case where $K=SU(2)$ acts on $ \text{SHerm}(2)$ by conjugation, this generalization reduces to the situation of \cite{roslerRauhut}, which can be seen as follows: We identify $\a = \text{SHerm}(2)\cong\R^{2}_{0}\cong\R$, so that $W=\{\pm 1\}, \pi(t)=2t, \Delta(t)=e^{it}-e^{-it}$ and $\overline{\a_{+}} = [0,\infty[$. Using these identifications and the fact that $K$ is of rank 1, the function $T$ is simply the indicator function $\chi_{[0,\infty[}$. Thus \cref{prop:dichteTranslation} reduces to
\begin{align*}
    (\delta_{r}*\delta_{s})(f)&=\frac{1}{2rs}\int_{0}^{\infty}f(t)t \sum_{i, j\in\{0, 1\}} (-1)^{i}(-1)^{j}T((-1)^{i}r+(-1)^{j}s - t)\ dt\\
    &=\frac{1}{2rs}\int_{|r-s|}^{r+s}f(t)t\ dt,
\end{align*}
which agrees with formula $(4.1)$ of \cite{roslerRauhut}. A direct calculation shows that ${P}_{+} = \frac 1 2\N_{0}$, so that $S_{\lambda}(t)=\sqrt{2}\Delta(t)^{-1}\sin((k+1)t)$ for $\lambda=\tfrac 1 2 k\in P_{+}$. 
This agrees, up to the factor $\Delta(t)^{-1}$, with the definition in \cite{roslerRauhut}, \S 4. In fact all our statements in this paper, such as for example the definition of the set $X$ (called $S$ in \cite{roslerRauhut}), are modified by this factor. This is a technical modification to accomodate for the fact that for non-semisimple Lie groups the Weyl vector $\rho$ is not necessarily contained in the weight lattice $P$. 
By re-introducing the factor $\Delta^{-1}$ we compensate a possible loss of $I$-periodicity of the functions $\Delta S_{\lambda}$.

It is a specialty of this rank $1$ situation that the weight lattice and the integral lattice align nicely, which made the explicit construction of orthonormal wavelets in \cite{roslerRauhut} work.


\begin{thebibliography}{999999}
\bibitem[BSO05]{BSO05} S. Ben Said, B. \O rsted, Analysis on flat symmetric spaces. {\it J. Math. Pures Appl. } 84 (2005), 1393--1426. 
\bibitem[BMM99]{baggett1999generalized} L.W. Baggett, H.A. Medina, K.D. Merrill, 
Generalized multi-resolution analyses and a construction procedure for all wavelet sets in $\mathbb R^n.$ {\it J. Fourier Anal. Appl.} 5 (1999), 563--573. 
\bibitem[Bum13]{bump} D. Bump, {\it Lie Groups.} Graduate Texts in Mathematics, Springer-Verlag, New York, 2013. 
\bibitem[Dau92]{daubechies1992ten} I. Daubechies, {\it Ten Lectures on Wavelets.} SIAM, 1992.
\bibitem[Far08]{faraut_2008} J. Faraut, {\it Analysis on Lie Groups: An Introduction.} Cambridge studies in advanced mathematics, Cambridge Univ. Press, 2008. 
\bibitem[FNS18]{FNS18} W. Freeden, M.Z. Nashed, M. Schreiner, {\it Spherical Sampling.} Birkh\"auser/Springer, Cham, 2018. 
\bibitem[FFP16]{FFP16} H.G. Feichtinger, H. F\"uhr, I.Z. Pesenson, Geometric space-frequency analysis on manifolds. {\it J. Fourier Anal. Appl.} 12 (2016), 1294--1355. 
\bibitem[GV88]{gangolliVaradarajan} R. Gangolli, V.S. Varadarajan, {\it Harmonic Analysis of Spherical Functions on Real Reductive Groups.}  Ergebnisse der Mathematik und ihrer Grenzgebiete, Springer-Verlag, Berlin, 1988.
\bibitem[GL95]{graczykLoeb} P. Graczyk, J.J. Loeb, Spherical analysis and central limit theorems on symmetric spaces. In: {\it	Probability Measures on Groups and Related Structures XI, ed. H. Heyer (Oberwolfach, 1994)}, pp. 146--166, World Scientific, 1995. 
\bibitem[GS02]{graczykSawyer_complex} P. Graczyk, P. Sawyer,  The product formula for the spherical functions on symmetric spaces in the complex case. {\it Pacific J. Math. } 204 (2002), 377--393. 
\bibitem[GS16]{graczykSawyer_survey} P. Graczyk and P. Sawyer, Convolution of orbital measures on symmetric spaces: a survey, in {\it Probability on algebraic and geometric structures}, 81--110, Contemp. Math., 668, Amer. Math. Soc., Providence, RI, 2016
\bibitem[HC57]{harish1957differential} Harish-Chandra, Differential operators on a semisimple Lie algebra. {\it Amer. J. Math. } 79 (1957), 87--120.
\bibitem[Hel00]{helgasonGroups} S. Helgason, {\it Groups and Geometric Analysis: Integral Geometry, Invariant Differential Operators, and Spherical Functions.}  Mathematical Surveys and Monographs, Amer. Math. Soc., reprint, 1984. 
\bibitem[Hum90]{Hum90} J.E. Humphreys, {\it Reflection Groups and Coxeter Groups.} Cambridge studies in advanced mathematics 29, Cambridge Univ. Press, 1990.
\bibitem[IZ80]{itzyksonZuber} C. Itzykson, J.B. Zuber, The planar approximation II. {\it J. Math. Phys.} 21 (1980), 411--421. 
\bibitem[Jew75]{jewett1975spaces} R.I. Jewett, Spaces with an abstract convolution of measures. {\it  
Adv. Math. } 18 (1975), 1--101. 
\bibitem[Kna96]{knapp} A.W. Knapp, {\it Lie Groups Beyond an Introduction.} Progress in Mathematics, Birkh\"auser  Boston, 2nd ed. 2005. 
\bibitem[Mad93]{madych1993some} W.R. Madych, Some elementary properties of multiresolution analyses of $L^2(\mathbb R^n).$ {\it Wavelet Anal. Appl.} 2, pp. 259--294, Academic Press, Boston, 1992. 
\bibitem[Mal09]{Mal09} S. Mallat, {\it A Wavelet Tour of Signal Processing.} Elsevier/Academic Press,  Amsterdam, 2009. 
\bibitem[McS21]{mcswiggen} C. McSwiggen, The Harish-Chandra integral: An introduction with examples. {\it Enseign. Math.} 67 (2021), 229--299. 
\bibitem[OOR06]{OOR06} G. Olafsson, E. Ournycheva, B. Rubin, Higher-rank wavelet transforms, ridgelet transforms, and Radon transforms on the space of matrices. {\it Appl. Comput. Harmon. Anal.} 21 (2006), 182--203. 
\bibitem[Pap11]{Pap11} M. Pap, Hyperbolic wavelets and multiresolution in $H^2(\mathbb T)$. {\it J. Fourier Anal. Appl.} 17 (2011), 755--776.
\bibitem[Pro07]{procesi} C. Procesi, {\it Lie Groups: An Approach through Invariants and Representations.} Springer-Verlag, 2007. 
\bibitem[RR03]{roslerRauhut} H. Rauhut, M. R\"osler, Radial multiresolution in dimension three, {\it Constr. Approx. } 22 (2003), 193--218. 
\bibitem[RV08]{roslerVoitAngers} M. R\"osler, M. Voit, Dunkl theory, convolution algebras, and related Markov processes, {\it In: Harmonic and Stochastic Analysis of Dunkl Processes}, eds. P. Graczyk et al., Travaux en cours 71, pp. 1--112, Hermann Mathematiques, Paris, 2008. 
\bibitem[Woj97]{wojtaszczyk} P. Wojtaszczyk, {\it A Mathematical Introduction to Wavelets.} London Mathematical Society Student Texts 37, Cambridge Univ. Press, 1997. 
\bibitem[Wol06]{wolf2006spherical} J.A. Wolf, Spherical functions on Euclidean space. {\it 
J. Funct. Anal.} 239 (2006), 127--136. 
\end{thebibliography}
\end{document}